\documentclass[onefignum,onetabnum]{siamart190516}

\usepackage{lipsum}
\usepackage{amsfonts}
\usepackage{graphicx}
\usepackage{epstopdf}
\usepackage{algorithmic}
\ifpdf
\DeclareGraphicsExtensions{.eps,.pdf,.png,.jpg}
\else
\DeclareGraphicsExtensions{.eps}
\fi

\newsiamremark{remark}{Remark}
\newsiamremark{hypothesis}{Hypothesis}
\crefname{hypothesis}{Hypothesis}{Hypotheses}
\newsiamthm{claim}{Claim}
\newsiamthm{modproblem}{Problem}
\newsiamthm{problem}{Problem}

\headers{Iterative Computation  of  SOS}{B. Despr\'es and M. Herda}

\title{Computation of sum of squares polynomials from data points} 

\author{Bruno Despr\'es\thanks{Sorbonne-Universit\'e, CNRS, Universit\'e de Paris, Laboratoire Jacques-Louis Lions (LJLL), F-75005 Paris, France,
        and  Institut Universitaire de France} \and Maxime Herda\thanks{Inria, Univ. Lille, CNRS, UMR 8524 – Laboratoire Paul Painlev\'e, F-59000 Lille, France}}

\usepackage{amsopn}

\newcommand\bX{\mathbf{X}}
\newcommand\cV{\mathcal{U}}
\newcommand\bx{\mathbf{x}}
\newcommand\by{\mathbf{y}}

\newcommand\bU{\mathbf{U}}
\newcommand\bV{\mathbf{V}}
\newcommand\bd{\mathbf{d}}
\newcommand\lla{\left\langle}
\newcommand\rra{\right\rangle}
\newcommand\RR{\mathbb{R}}
\newcommand\cD{\mathcal{D}}
\newcommand\ds{\displaystyle}
\newcommand\nrm{\|}
\newcommand\sP{\mathrm{P}}
\newcommand\eps{\varepsilon}

\newcommand\blue[1]{{\color{black}#1}}

\usepackage{amsmath, amssymb, mathrsfs,mathabx}
\usepackage{xcolor}
\usepackage{graphicx}
\usepackage{enumitem}
\usepackage{float}
\usepackage{tikz, pgfplots}
\usepgfplotslibrary{groupplots}
\usepackage[scale = .73]{geometry}

\begin{document}
    
    \maketitle
    
    \begin{abstract}
        We propose an iterative  algorithm for the numerical computation of sums of squares of polynomials approximating given data at prescribed interpolation points. The method is based on the definition of a convex functional $G$ arising from the dualization of a quadratic regression over the Cholesky factors of the sum of squares decomposition. In order to justify the construction, the domain of $G$, the boundary of the domain and the behavior at infinity are analyzed in details.  When the data interpolate a positive univariate polynomial, we show that in the context of the Lukacs sum of squares representation, $G$ is coercive and strictly convex which yields a unique critical point and a corresponding decomposition in sum of squares.  For multivariate polynomials which admit a decomposition in sum of squares and  up to a small perturbation of size $\varepsilon$, $G^\varepsilon$ is always coercive  and so it minimum yields an approximate decomposition in sum of squares. Various unconstrained descent algorithms are proposed to minimize $G$. Numerical examples are provided, for  univariate and bivariate  polynomials.
    \end{abstract}
    
    

    \section{Introduction}
    
     The  numerical and algorithmic motivation of the present paper comes from a   recent work  \cite{charli} 
    {where an  iterative algorithm for  positive interpolation 
    (meaning  that a sign condition on a given closed interval  $\mathbb I$ must be respected)
    was proposed for of univariate polynomials.
     A practical scenario which illustrates the interest of iterative positive interpolation
    is the following.
    Take a polynomial without knowing its sign  on $\mathbb I$.
    If  the iterative method    converges and recover $p$ at the limit (it can be checked at a finite number of points), 
    then $p$ is non negative on $\mathbb I$ (that is for an infinite number of points).
    In this case the algorithm provides an iterative {\it certificate of positivity}
   \cite{lasserre_2010_moments, lasserre_2015_introduction}. But if the iterations do not recover $p$ at the limit (or
    if one stops the algorithm after a finite number of iterations),
     then   $p$ is (or might be) non positive on $\mathbb I$.  In  this case of non convergence, the iterations provide nevertheless a non negative surrogate to $p$. We refer to the quoted work for an illustration
    of the interest of  non negative polynomial surrogates in the context of  Scientific Computing (SC).}
    However 
    two important restrictions in  the previous algorithm  \cite{charli} are that  the polynomials are univariate and the interpolation points, where the data of the polynomials are given, are sliding points (it allowed for strong convergence properties). It  brings severe constraints  for applications in SC.  
    In the present work, we relax these restrictions by constructing a  new iterative algorithm for positive interpolation.
    The algorithm   aims at computing  a sum of squares (SOS) decomposition from the sole knowledge
    of prescribed interpolation data at prescribed interpolation points. Also the method is much more general so it is formulated for multivariate polynomials as well and does not need tensorization, something that was impossible with the previous method.

    A  modern  reference in SC  for  control of the sign of polynomials at a \emph{finite number of   prescribed interpolation points}  is in  the works of C.-W. Shu  \cite{shu}, with application to the discretization of hyperbolic equations with high order methods.  The point of view developed in this article is to control the sign of polynomials on \emph{all points in a given compact (semi-algebraic) set $\mathbb K\subset \mathbb R^d$} which is much more demanding. Preliminary tests for the construction of such algorithms are in \cite{despres_2017_polynomials}, but the methods were inefficient in terms of the time of restitution. In a fully different direction, one must mention the theory of numerical approximation with  splines, see \cite{lee,veiga}: splines are widely used in scientific computing and computer aided design (CAD) but often needs tensorization in multi-dimension; this limitation is not encountered by our new methods because they can be implemented
    on any semi-algebraic set $\mathbb K$ in any dimension.
    
    In the community of numerical optimization \cite{lasserre_2010_moments} from which we borrow most of our notations, SOS algorithms based on SemiDefinite Programming (SDP) are extensively used. It had been noticed by Powers and W\"ormann \cite{powers_1998_algorithm} that finding an SOS decomposition is equivalent to SDP, that is optimization in the cone of non-negative quadratic forms. Then algorithms based on interior-point methods were developed to solve these problems \cite{nesterov,nesti1,nesti2}. However, these methods seem to be hardly directly applicable in SC because they are based more on algebraic properties and not on interpolation data which are of major importance  in numerical analysis and SC. This leads us to the development of the algorithm of the present paper, which is not based on SDP but rather on the iterative resolution of a non-convex quadratic problem over Cholesky factors of the SOS decomposition.  We solve the quadratic program using a dualization of the problem, which leads us to a nonlinear convex program.  Let us mention that a similar reformulation of general SDP was proposed by Burer and Monteiro \cite{burer_2003_nonlinear}. In our case however, we use the particular structure of the interpolation data of the 
    SOS  to obtain some useful coercivity properties on the dual function. Also, similar dualization ideas can be found in \cite{malik,henrion_2011_projection}, but unlike here they are formulated on the Gram matrix rather than on the Cholesky factors.
    Our  construction  will generate a functional with strong convexity properties for which standard descent algorithms are efficient, as shown in the numerical 
    section.
    
    Let $\sP[\bX]:=\sP[X_1, \dots X_d]$ be the set of real polynomials with $d$ variables.
    The subset of polynomials of total degree less than or equal to $n\geq 1$ is denoted by $\sP^n[\bX]$,
    with
    $       r_*=\mbox{dim } \sP^n[ \mathbf X]
$.
    Let $\mathbb{K}\subset \mathbb R$ be a closed  semi-algebraic set defined through  a finite number
    $j_*$  of polynomial inequalities
    \begin{equation} \label{eq:1}
        \mathbb{K} =\left\{
        \mathbf x\in \mathbb R^d \text{ such that } g_j(\mathbf x)\geq 0\mbox{ for }g_j\in \sP[\mathbf X] ,  1\leq j \leq j_*
        \right\}.
    \end{equation}
    Most standard cells (intervals in 1D, squares and triangles in 2D, \dots) in SC can be implemented as semi-algebraic sets, so it is not a restriction for further applications.
    The convex set of non-negative polynomials of maximal degree $n$ on $\mathbb K$  is 
    \begin{equation} \label{eq:2}
        \sP_{\mathbb{K},+}^n[\mathbf X]=\left\{
        p\in \sP^n[\mathbf X] \text{ such that } p(\mathbf x)\geq 0 \mbox{ for any }\mathbf x\in K
         \right\}.
            \end{equation}
    Famous examples of characterizations as SOS are the {Lukacs theorem} \cite{szego_1975_orthogonal} or {Putinar's Positvstellensatz} \cite{putinar_1993_positive}: a recent state of the art can be found in the books of Lasserre \cite{lasserre_2010_moments, lasserre_2015_introduction};
    some recent algorithmic issues in the context of optimal control can be found in \cite{henrion} and therein. In order to be constructive, we focus in this work on the following version
    \begin{equation}     \label{e:sos}
        p= \sum_{j=1}^{j_*}g_j \left(\sum_{i=1}^{i_*}q_{ij}^2\right) =
        \sum_{i=1}^{i_*} \left(    \sum_{j=1}^{j_*}g_jq_{ij}^2 \right) = \sum_{j=1}^{j_*}\sum_{i=1}^{i_*} g_j q_{ij}^2,
    \end{equation} 
    where the maximal number of squares is  equal to a predefined value $i_*\geq 1$ independent of $j$. 
      {In this work, the number of squares $i_*$ and the degree of the polynomials $q_{ij}$ are prescribed in function of $n$, 
    see below (\ref{eq:iegr}) for the prescription on  $i_*$ and (\ref{eq:degqij}) for the prescription the degree of the polynomials $q_{ij}$.
   It can be compared with the   Schm\"ugden's or  {Putinar's Positvstellensatz}  where
   the degree of the polynomials $q_{ij}$ can be exponentially large \cite{nie,lasserre_2010_moments}. With our notations, 
   it  is sufficient to embed $p$ in a set of polynomials of larger degree, that is to say to take $n\gg\mbox{deg}(p)$,
    to recover this case. }

    Next, the notion of unisolvence which comes from the {Finite Element Method} (FEM) is convenient to formalize properties of interpolation points. A unisolvent set of points 
    $\left(\mathbf x_r\right)_{1\leq r \leq r_*}$
    is  such that any polynomial   $p\in \sP^n[ \mathbf X]$ is uniquely determined by its values 
    $y_r=p(\mathbf x_r)$ for $1\leq r \leq r_*$.
      The number $i_*$ of  polynomials in the SOS (\ref{e:sos}) is  \emph{a priori} independent from the number of interpolation points. However in our context the function $G$ below is more naturally constructed assuming that 
    \begin{equation} \label{eq:iegr}
        i_*=r_*.
    \end{equation}
    That is why we  will assume (\ref{eq:iegr}) throughout this work, except at  early stages of  the construction.
     With these notations, one  formulates the notion of positive interpolation: it  is a recent  adaptation \cite{charli} to SC of  the notion  of a {\it certificate of positivity} for which the reader can find  information in \cite{lasserre_2010_moments, lasserre_2015_introduction}.
    A practical way to understand the model problem below is the following: from the knowledge of the values of $p$ 	at only a \emph{finite number of given interpolation points}, get a control of the sign of $p$ at \emph{infinite number of points} (the whole set $\mathbb K$).

    \begin{modproblem}[Iterative positive interpolation on $\mathbb K$]\label{p:main} 
        Let  $p\in \sP^n_{\mathbb{K},+}[ \mathbf X]$.
        Take a unisolvent set  $\left(\mathbf x_r\right)_{1\leq r \leq r_*}$, and consider  the interpolated values $y_r=p(\mathbf x_r)$. From $(\bx_r,y_r)_{1\leq r \leq r_*}$, compute  iteratively   polynomials $(q_{ij})_{ij}$ such that the SOS representation \eqref{e:sos} holds at the limit.
    \end{modproblem}

    The   methods and results studied in this work can be summarized as follows. 
     Consider the parametrization 
    \begin{equation} \label{eq:degqij}
    q_{ij}\in \sP^{n_j}[\mathbf X] \mbox{ with }
    n_j = \left\lfloor(n-{\rm deg} ( g_j))/2 \right\rfloor
    \end{equation}
    where $\lfloor\cdot\rfloor$ denotes the integer part of a real number. Consider  the canonical  basis made of monomials (but other basis can be taken as well, see Remark~\ref{rem:monom}), with the standard multi-index notation   $\alpha=(\alpha_1, \dots, \alpha_d)\in \mathbb N^d$,  $|\alpha|=\alpha_1+\dots +\alpha_d$ and $\bX^\alpha=X_1^{\alpha_1}\dots X_d^{\alpha_d}$.   The polynomials $q_{ij}$ write 
    $  q_{ij}(\bX)= \sum_{|\alpha|\leq n_j } c_\alpha^{ij} \mathbf X^\alpha$ and we store the coefficients in a vector of coefficients $c^{ij}=(c^{ij}_\alpha)_\alpha\in \RR^{r_{j}}$
    where  
    $
    r_{j}=
    \mathrm{dim}(\sP^{n_j}[\bX]) = \binom{d+n_j}{d}$.
    Gather the coefficients $c^{i1}$, $c^{i2}$, ..., $c^{ij_*}$  in a single column vector (called a Cholesky factor)
    $    \bU_i= \left(
    c^{i1}, c^{i2}, \dots, c^{ij_*}
    \right)^t \in \mathbb R^{r_*} $ where $r_* = \sum_{j=1}^{j_*} r_{j}$.
    {Define the Hankel matrices
    $ 
        D_{\alpha,\beta}^{n_j}(\bX) =  \bX^\alpha \bX^\beta$ for $  |\alpha|, |\beta| \leq n_j$. Define  
         the  polynomial valued block matrix $B(\bX)=B(\bX)^t \in \mathbb R^{r_*\times r_*}$ 
    \begin{equation}
        B(\bX) = \mathrm{diag}\left(
        g_1(\bX) D^{n_1}(\bX), 
        ,\dots, g_{j_*}(\bX) D^{n_{j_*}}(\bX)
        \right) .
        \label{e:defB}
    \end{equation}
    This matrix is a block diagonal localizing matrix \cite{lasserre_2010_moments}.}
    The first diagonal block  is square $r_1\times r_1$, 
    \dots 
    until the last block which is square $r_{j_*}\times r_{j_*}$:
    all other terms are zero. By construction, one has the identity
    \begin{equation} \label{e:defB2}
        \sum_{j=1}^{j_*}g_j(\bX) \sum_{i=1}^{i_*}q_{ij}^2(\bX) =
        \sum_{i=1}^{i_*}
        \left(  \sum_{j=1}^{j_*}g_j(\bX)q_{ij}^2(\bX) \right)
        = \sum_{i=1}^{i_*}\lla B(\bX) \mathbf U_i, \mathbf U_i \rra .
    \end{equation}
    Denote the evaluation  of $B(\mathbf X)$  at interpolation points as
    $    B_r = B(\bx_r)\in \mathbb R^{r_*\times r_*}$.    
    Define the function 
    $G :  \mathbb R^{r_*}  \to \blue{ \overline{ \mathbb R }=\mathbb{R}\cup\{+\infty\}}$ with domain  $
    \cD=
    \left\{
    \lambda\in \mathbb R^{r_*}\mbox{ such that } I+ \sum_{r=1}^{r_*} \lambda_r B_r \blue{\succ}0
    \right\}$ as follows. 
    For $\lambda\in \cD$ then
    \begin{equation} \label{eq:gigi}
        G(\lambda)= \mbox{tr}\left[ \left(I+\sum_{r=1}^{r_*} \lambda_r B_r \right)^{-1}\right] + \sum_{r=1}^{r_*}  y_r \lambda_r,
    \end{equation}
    otherwise
    $G(\lambda)=+\infty$. In the previous formula, $\mbox{tr}(\cdot)$ denotes the trace. 
    Our main results are the following.

    \begin{theorem} \label{t:main}     The function $G$ has the following properties:
        \begin{enumerate}
            \item[1.] 
            It  is a proper closed convex function. It is $C^\infty$ on its non-empty open convex domain $\mathcal D$,
            tends to infinity at $\partial \mathcal D$ and is  infinite  everywhere else by definition. 
            
            \item[2.]  
            Each $\lambda\in\cD$    defines  {computable}  polynomials $(q_{ij}[\lambda])_{1\leq i\leq r_*,1\leq j\leq j_*}$ such that 
            \begin{equation} \label{eq:gragra}
                \dfrac{\partial G}{\partial{\lambda_r}} (\lambda)  = y_r - \sum_{j=1}^{j_*}g_j(\bx_r) \sum_{i=1}^{r_*}q_{ij}^2[\lambda](\bx_r), \quad 1\leq r\leq r_*.
            \end{equation}
            If $\lambda_*\in  \mathcal D$ is a critical point of $G$, that is $\nabla G(\lambda_*)=0$, then the family 
            $(q_{ij}[\lambda_*])_{ij}$  is solution to (\ref{e:sos}), that is a SOS.
            
        \end{enumerate}
    \end{theorem}

    \begin{theorem}[Existence of critical points in $\cD$] \label{th:th2}
        It  is proved in  two cases.
        \begin{enumerate}
            \item[1.] Take  $d> 1$, $\mathbb K$ a semi-algebraic set and   $p\in  \sP_{\mathbb{K},+}^n[\mathbf X]$. Assume that a technical condition on the linear independence
            of the matrices $B_r$ is satisfied. Then, up to an infinitesimally small perturbation
            (the perturbed polynomial $p^\varepsilon$ has 
            the interpolation data  $(y_r^\varepsilon)_{1\leq r \leq r_*}$), the function $G^\varepsilon$ is strictly convex, coercive and  admits a unique critical point in $\cD$.
            \item[2.] Take $d=1$, $\mathbb{K}$  a segment and $p>0$   on  $\mathbb{K}$. 
            Then  the technical condition the linear independence
            of the matrices $B_r$ is satisfied. Moreover  $G$ is strictly convex, coercive and  admits a unique critical point in $\cD$.
        \end{enumerate}
    \end{theorem}
    
    \begin{corollary}[Solution to Problem \ref{p:main}] \label{cor:cori1}
        Under the hypothesis of Theorem \ref{th:th2}, the minimum of $G$ (or $G^\varepsilon$) in $\cD$ yields a  SOS decomposition of  $p$ (or $p^\varepsilon$). It can be computed by standard descent algorithms.
    \end{corollary}
    
    As stated in the introduction, a practical scenario which in our mind has interest for SC is the following.
    Take a polynomial without knowing its sign  on $\mathbb K$.
    If  the descent method    converges and recover $p$ at the limit, then $p$ is non negative on $\mathbb K$. If the descent does not recover $p$ at the limit, then  for monovariate polynomials, $p$ is non positive on $\mathbb K$. It shows that the descent method provides an iterative {\it certificate of positivity}. In  case of non convergence, the iterations provide nevertheless a non negative surrogate to $p$. We refer to \cite{charli} for an illustration
    of the interest of  non negative polynomial surrogates.
    
    The outline of this paper is as follows. In Section~\ref{s:dual}, we propose a dual interpretation of Problem \ref{p:main}. This leads us to the introduction of the function $G$ and its domain.
     Then, in Section~\ref{s:properties}, we discuss necessary and sufficient conditions characterizing asymptotic properties and strict convexity of $G$. In Section~\ref{s:univariate}, we show that for univariate positive polynomials on a segment, the former conditions are satisfied yielding strict convexity and coercivity of the associated function $G$. Besides, we provide a more precise description the structure of the domain. In Section \ref{s:methods}, we present the specific descent and Newton type methods we use to compute the critical points of $G$. In Section~\ref{s:numerics} we provide numerical illustrations of the efficiency of our new approach  for computing SOS decomposition of polynomials in one variable on segment and two variables on triangle. Finally, we provide in Appendix~\ref{s:asympcone} some additional theoretical results concerning the links between the asymptotic cone of the set $\mathcal{D}$ and the Lagrange polynomials in the case of univariate polynomials.
    \newline
    {\it Acknowledgements.}
    Both authors are greatly indebted to  Jean-Bernard Lasserre and Didier Henrion for their kind explanations on the theory and state of the art of semidefinite programming and sum of squares and would like to thank them for their invitation at LAAS and for their hospitality. \blue{The authors would also like to thank the anonymous referees for their suggestions and comments which helped to improve the quality of this paper.}

    \section{Construction of $G$ (Proof of Theorem \ref{t:main})}\label{s:dual}
    
    The construction of $G$, leading to \eqref{eq:gigi}, is done by  recasting the   model problem \ref{p:main} as  the convex dual of a Quadratically Constrained Quadratic Program (QCQP) (see\cite{boyd_2004_convex}).  In order to have a more general discussion, we relax the condition (\ref{eq:iegr})  in this part and in the next Section \ref{sec:ie}. It means that 
    $$
    i_*\neq r_*
    $$
    is possible as well.
     The condition (\ref{eq:iegr}) is reintroduced end of Section \ref{sec:ie}.
    We begin with some remarks on the objects introduced in the first section.

    \begin{remark}\label{rem:monom}
        In the numerical experiments of Section \ref{s:numerics}, we use other polynomials  than the monomials  in order to optimize the robustness and accuracy  of the algorithms.
        It only changes the definition  of the matrix $B(\mathbf{X})$ in \eqref{e:defB} and thus of $B_r = B(\mathbf{x}_r)$ but every result of this paper still hold. More generally, one could even generalize Problem~\ref{p:main} and replace the constraint of equality at interpolated values by constraints of the type $y_r = L_r(p)$ where the family $\{L_r:P^{n}[\mathbf{X}]\rightarrow\mathbb{R},\ r = 1,\dots,r_*\}$ is any basis of the dual space of $P^{n}[\mathbf{X}]$. In the context of SC  more precisely for the numerical resolution of partial differential equations, one deals with data points in finite difference discretizations. However, if one is considering a finite volume discretization, one would rather work with mean values on some mesh cells. This variant is easily manageable with our method by choosing the adequate linear forms $L_r$ and modifying the matrices $B_r$ accordingly.
    \end{remark}

%

    \begin{remark} \label{anc:lemma2.10}
        An interesting consequence of the  Caratheodory Theorem (\cite[Theorem III.1.3.6 page 98]{hiriart_1993_convex})
        is that if     a formula like (\ref{e:sos}) holds for $i_*>r_*$, then a similar one holds
        also for $i_*=r_*$ (but for different polynomials $q_{ij}$). 
        Indeed the set
        $
        \mathcal W= \sum_{j=1}^{j_*} g_j(\mathbf X) \sP^{n_j}[\mathbf X]^2
        $
        is a closed convex cone embedded in $ p\in \sP^n[\mathbf X]$.
        Therefore any convex combination of $i_*>r_*$ elements of $\mathcal W$ can be expressed as a convex
        combination
        of only  $r_*=\mbox{dim }\sP^n[\mathbf X]$ elements of  $\mathcal W$ (the coefficients of the convex combination
        can be set to 1 after proper rescaling of the new $q_{ij}$).
    \end{remark}

    \subsection{Lagrangian duality (Theorem \ref{t:main} item 1)} \label{sec:ie}

    In any dimension, the notation   $\lla\cdot,\cdot\rra$ will denote the Euclidean dot product and   $\|\cdot\|$ will denote  the associated norm. We define  the algebraic manifold 
    \begin{equation} \label{eq:aze2}
        \cV=\{\bU=(\bU_1, \dots, \bU_{i_*})\in
        \left(  \mathbb{R}^{r_*} \right)^{ i_*}\text{ such that } \sum_{i=1}^{i_*}\lla B_r \mathbf U_i, \mathbf U_i\rra
        =y_r
        \text{ for all }1\leq r\leq r_*\} .
    \end{equation}
    The vectors $\mathbf U_i$ are called the Cholesky factors \cite{nesti1,nesti2} of the decomposition.
    With the unisolvence assumption, finding a SOS \eqref{e:sos} amounts to finding one element 
    $\mathbf U\in \cV$.
    In order to find a $\mathbf U\in \cV$ in a constructive manner, our  strategy is to
    start  at a given $\mathbf V$ (probably outside $\cV$)
    and to project on $\cV$ in the quadratic norm. It writes as follows.
    
    \begin{problem} \label{prob:aze}
        Take  $\bV=(\bV_1, \dots, \bV_{i_*})\in
        \left(  \mathbb{R}^{r_*} \right)^{ i_*}$. Calculate $
        \mathbf U=   \underset{\bU\in\cV}{\mbox{argmin }   }
        \frac{1}{2}\sum_{i=1}^{i_*}\|\bU_i - \bV_i\|^2 $. 
    \end{problem}
    
    The vectors $\bV=(\bV_i)_i$ may be thought of  as a good initial guesses for the 
    $\bU=(\bU_i)_i$. The optimal value of the  cost  does not matter. However Problem \ref{prob:aze} seems even harder to solve than the original problem we were concerned with.
    The  finding is that the Lagrangian dual problem   is endowed with good properties provided
    $\mathbf V$ is conveniently chosen. In this case, the new Problem \ref{prob:aze} provides a 
    way to determine an admissible $\bf U\in \cV$. 
    
    Still for any $\mathbf V$, introduce the Lagrangian which is the sum of the functional  and of the dualization of the constraint (\ref{eq:aze2}) with a Lagrange multiplier $\lambda\in \mathbb R^{r_*}$
    \[
    \mathcal L(\bU, \lambda)= \frac12\sum_{i=1}^{i_*}\left(\|\bU_i - \bV_i\|^2  + \sum_{r=1}^{r_*} \lambda_r\lla B_r \mathbf U_i, \mathbf U_i \rra\right) -\frac12 \lla\lambda,\by\rra  \qquad \mbox{where }
    \mathbf y =\left(
    y_r\right)_{1\leq r \leq r_*} .
    \]
     The first-order optimality constraints are  $\nabla_{ \mathbf U} \mathcal L=0$ and $\nabla_ \lambda \mathcal L=0$. 
     The first-order optimality constraint  $\nabla_{ \mathbf U} \mathcal L=0$  is 
     linear with respect to $\mathbf U$. Define  the symmetric matrix $M(\lambda) =M(\lambda)^t \in \mathbb R^{r_*\times r_*} $ 
    \begin{equation}\label{e:mlambda}
        M(\lambda) =I+\sum_{r= 1}^{r_*} \lambda_r B_r
    \end{equation}
    where $I$ is the identity matrix in $ \mathbb R^{r_*\times r_*}$.
    The  condition   $\nabla_{ \mathbf U} \mathcal L=0$ writes
    \begin{equation} \label{eq:vtou}
        M(\lambda)\mathbf U_i = \mathbf V_i \mbox{ for }1\leq i \leq i_*
        \Longleftrightarrow 
        M(\lambda)\mathbf U = \mathbf V.
    \end{equation}
    If the multiplier  $\lambda \in \mathbb R^{r_*}$ is such that 
    the matrix $M(\lambda)$ is invertible,  then the candidate solution $\mathbf U$  can be computed 
    explicitly in terms of $\lambda$ and $\mathbf V$ as the solution of the linear system (\ref{eq:vtou}).
    
    It is therefore natural to concentrate on a condition on $\lambda$ such that $M(\lambda)$ is invertible. In order to obtain convexity properties in the following we even restrict $\lambda$ to the set of positive definiteness of $M(\lambda)$. 
    To our knowledge, this at this stage that our analysis differs from the standard exposition of dual QCQP \cite{boyd_2004_convex,hiriart_1993_convex} and from other dualizations in the context of SOS \cite{burer_2003_nonlinear,malik,henrion_2011_projection}.

    \begin{definition}\label{d:D}
        The domain of positive definiteness of $M$ is 
        $      \mathcal D=\left\{
        \lambda\in \mathbb R^{r_*}\mid M(\lambda)\blue{\succ} 0 
        \right\} \subset \mathbb R^{r_*} $. 
        It is  an open set and it is non empty  since $0\in \cD$.
    \end{definition}
    
    For a Lagrange multiplier $\lambda\in\cD$, the inverse transformation of (\ref{eq:vtou}) 
    is
    $
    \bU(\lambda) = M(\lambda)^{-1}\bV$. 
    Then, one can evaluate the Lagrangian at $\bU(\lambda)$.
    An elementary computation yields
    $
    \mathcal L(\bU(\lambda), \lambda)= \frac12\sum_{i=1}^{i_*}\left(\|\bV_i\|^2   -\lla \bV_i, M(\lambda)^{-1}\bV_i\rra\right)  -\frac12\lla\lambda,\by\rra $. 
    This motivates the introduction of the dual objective function $G_\bV:\mathcal D \longrightarrow \mathbb R$ defined by
    \begin{equation} \label{e:GV}
        G_\bV(\lambda)=\sum_{i=1}^{i_*}\lla \mathbf \bV_i, M(\lambda)^{-1}\bV_i \rra  +
        \lla \lambda, \mathbf y \rra\,,
    \end{equation}
    and which one should think of as a function to be minimized.
    \begin{lemma}\label{p:derivatives}
        The function $G_\bV$  is smooth on $\cD$. The first and second derivatives are 
        \begin{equation} \label{e:derivatives}
            \ds\frac{\partial G_\bV}{\partial\lambda_r} (\lambda)= \ds y_r - \sum_{i=1}^{i_*}\left<\mathbf  U_i(\lambda),B_r \mathbf U_i(\lambda)\right> 
            \mbox{ and }	\ds\frac{\partial^2G_\bV}{\partial\lambda_r\partial\lambda_s} (\lambda) = \ds 2 \sum_{i=1}^{i_*} \left<  B_r  \mathbf  U_i(\lambda),M(\lambda)^{-1}  B_s \mathbf U_i(\lambda)\right> .
        \end{equation}
        In particular $G_\bV$ is convex on  $\mathcal D$.
    \end{lemma}
    
    \begin{proof} 
        The proof stems from the identity 
        $\partial_{\lambda_r} M(\lambda)^{-1}=- M(\lambda)^{-1}B_r M(\lambda)^{-1} $ and the symmetry of the various  matrices involved. Convexity follows from the positivity of $M(\lambda)$ and the expression of second derivatives yielding
        $
        \lla \nabla^2G_\bV(\lambda) \mu, \mu \rra = 2 \sum_{i=1}^{i_*}
        \lla A_i(\mu,\lambda) , M(\lambda)^{-1} A_i(\mu,\lambda)  \rra \geq 0
        $
        where $ A_i(\mu,\lambda)=\sum_{r=1}^{r_*}\mu_r B_r \mathbf U_i(\lambda)$
        for $ 1\leq i \leq i_*$. 
    \end{proof}

    In order to address the behavior of $G_\bV$ near the boundary, we will make us of the following notion.
    
    \begin{definition}
        A convex function $f : \RR^{r_*}\mapsto \RR\cup\{+\infty\}$ is said to be closed  over its domain $\mathcal D_f=\left\{  \mathbf x\mid f(\mathbf x)<\infty  \right\}$ if and only if the level sets  $\left\{  \mathbf x\mid f(\mathbf x)\leq t  \right\}$ are closed for $t<+\infty$: see \cite{hiriart_1993_convex} or \cite[Appendix A.3.3.]{boyd_2004_convex}.
    \end{definition}
    This  property is extremely important in our approach because it yields a strong control of 
    the objective function at finite distance. 
    
    \begin{lemma}\label{p:closed}
        Assume the equality of dimensions  (\ref{eq:iegr}))\footnote{As a consequence, the notation $i_*$ will not be used anymore in the rest of the presentation, only $r_*$.}, that is $i_*=r_*$,     and that $\bV\in\RR^{r_*\times r_*}$ is an orthogonal matrix. Then one has the simpler
        expression 
        where $\mathrm{tr}(\cdot)$ denotes the trace of a square matrix
        \begin{equation} \label{eq:newgg}
            G(\lambda):=  G_\bV(\lambda)= \mathrm{tr} (M^{-1}(\lambda)) + \lla \lambda, \mathbf y \rra . 
        \end{equation}
        Moreover the extension of $G:=G_\bV$ with value $+\infty$ outside of $\cD$ is a closed convex function
        with open domain $\mathcal D$.
    \end{lemma}
    
    \begin{proof} 
        The formula is a direct consequence of (\ref{e:GV}), because the number $i_*$ of orthogonal vectors
        $\mathbf V_i$ is equal to the dimension $r_*$ of the space. Thanks to the continuity on $\cD$, the closedness of $G_\bV$ on $\RR^{r_*}$ amounts to showing that for any sequence $(\mu_k)_k$ in $\cD$ converging to a point of the boundary of the domain
        $
        \partial \mathcal D= \left\{ \lambda \in \overline{\mathcal D}\mid \mbox{det} \left(I+ \sum_{r= 1}^{r_*} \lambda_r B_r\right)=0 \right\} ,
        $ 
        then 
        one has $G_\bV(\mu_k)\to+\infty$ as $k\to+\infty$. In the light of the  representation formula 
        (\ref{eq:newgg})
        involving the trace of $M(\lambda)^{-1}$ it is  the case since the minimal eigenvalue of  $M(\mu_k)$ goes to $0$ as $k\to+\infty$.
        Clearly the function is independent of $\mathbf V$.
    \end{proof}

    { In order to have a better intuition of the structure of $G$,
 an illustration of its graph  is given on Figure  \ref{fig:2}.    Near the boundary of its domain, the function $G$ behaves by construction   like    a  rational barrier function \cite{nesterov,boyd_2004_convex}.
This barrier does not introduce any kind of approximation, and it   is an exact one.
This property is a strong  algorithmic asset of $G$ with respect to more standard logarithmic barrier methods.
The figure provides three different 
illustrations  of a closed convex function which is infinite outside of its domain.
Actually we will show in the sections below that $G$ is linear at infinity (in the direction of the asymptotic cone).
The whole point will be to understand under which conditions $G$ is coercive, which corresponds to
 the rightmost plot on Figure~\ref{fig:2}.
 It will  prove the existence of a multiplier in $\mathcal D$, without explicitly requiring to use the methods of Lagrangian duality.

  \begin{figure}[h]
    \begin{center}
      \begin{tabular}{ccc}
	\scalebox{0.28}{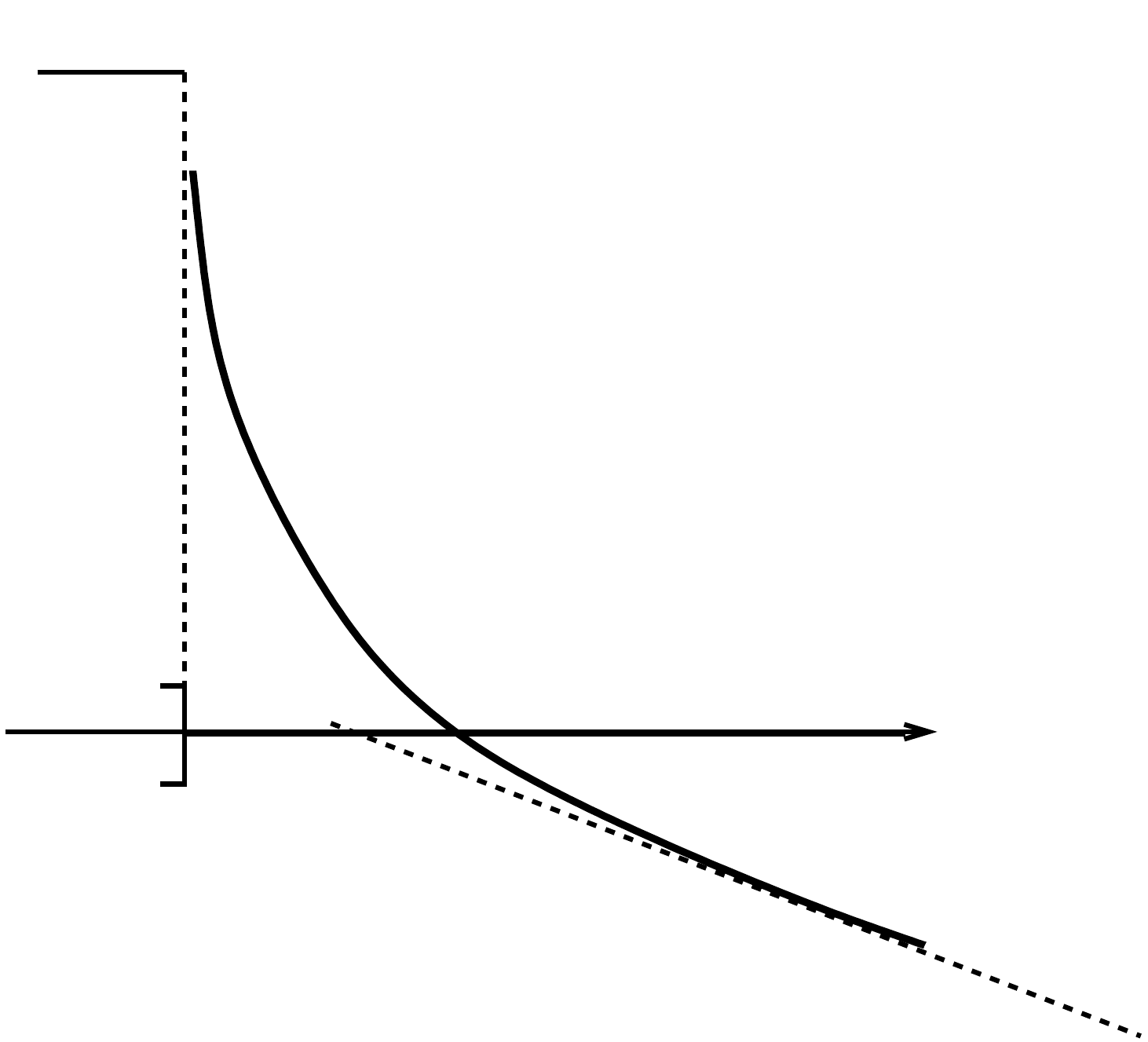} &
	\scalebox{0.28}{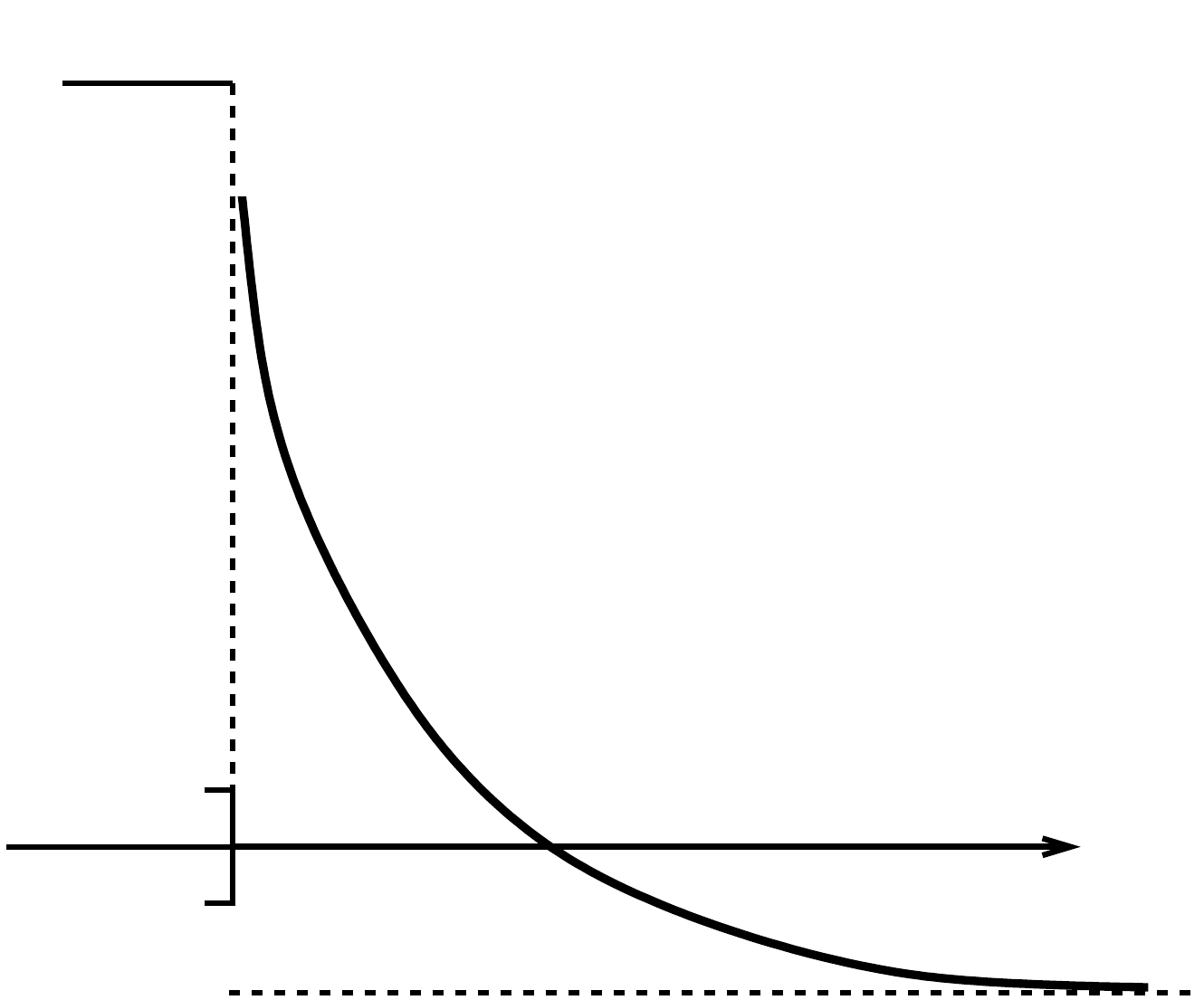} & 
	 \scalebox{0.33}{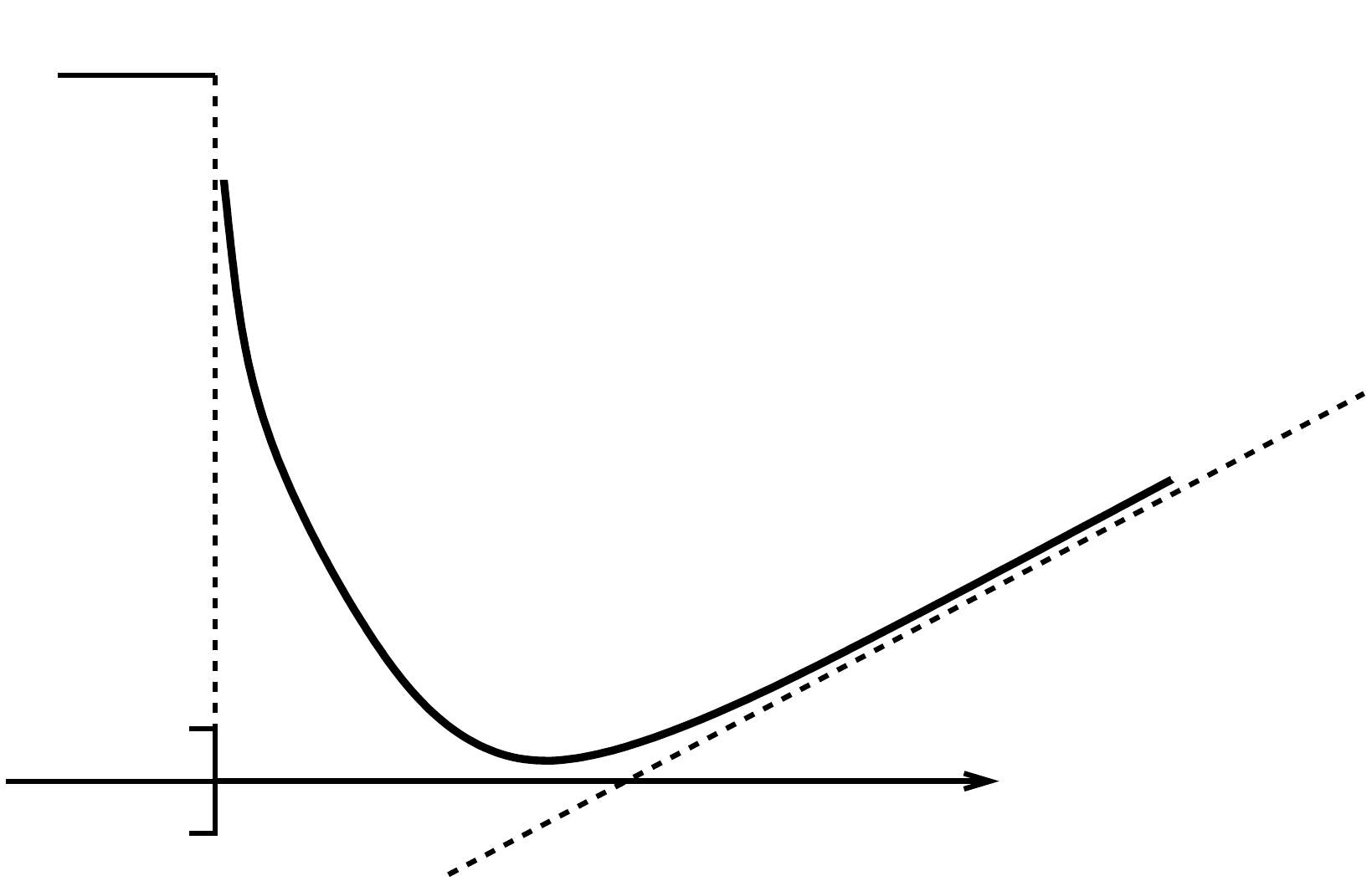} 
      \end{tabular}
    \end{center}
    \caption{Three cases of the graph of a closed convex  function $f$ which
      is convex over its open domain  and asymptotically linear at infinity.
       On the left, the function is not lower bounded and not coercive.
      In the center the function is lower bounded but not coercive.
      On the right, the function is lower bounded and coercive.}
    \label{fig:2}
  \end{figure}

 Even if the Lagrange multiplier $\lambda$ is constrained in the domain $\mathcal D$, the minimization of the dual function $G$ can be done by essentially unconstrained descent algorithms thanks to its coercivity properties. We will make this clearer in Section~\ref{s:methods}. This behavior is another asset of the function $G$.
}

    \subsection{Critical points of $G$  (Theorem \ref{t:main} item 2)}

      In this section, we formalize natural consequences of the formulas  \eqref{e:derivatives} for the derivatives of $G$, which 
    are preparatory to establish  that the $G$ is naturally coercive in the domain $\mathcal D$.
    These first properties are essentially a reformulation of the previous material. 

    These first properties are essentially a reformulation of the previous material. 
    For each Lagrange multiplier $\lambda\in\cD$ one defines the vectors 
    $(c_\alpha^{ij}[\lambda])_{\alpha,j}\in \mathbb R^{r_j}$ which are the components of $\bU_i(\lambda)$, the latter being the $i$th column of $M(\lambda)^{-1}$.
    It defines 
    the  polynomials 
    $q_{ij}[\lambda]\in \sP^{n_j}[\mathbf X]$ by 
    $
    q_{ij}[\lambda](\bX)= \sum_{|\alpha|\leq n_j } c_\alpha^{ij}[\lambda] \mathbf X^\alpha$. 
    With (\ref{e:defB2}), these polynomials define a sum of square $p[\lambda]\in  \sP_{\mathbb{K},+}^n[\mathbf X]$
    \begin{equation}
        p[\lambda](\bX)  =\sum_{j=1}^{j_*}g_j(\bX)  \left(  \sum_{i=1}^{r_*} q_{ij}^2[\lambda](\bX)\right).
        \label{e:plambda}
    \end{equation}
    Using  \eqref{e:defB2},  $p[\lambda](\bx_r)=\sum_{i=1}^{r_*}\left< B_r \mathbf U_i, \mathbf U_i\right> $. 
    So  \eqref{e:derivatives} is rewritten as 
    \begin{equation}
        \frac{\partial G}{\partial \lambda_r}(\lambda)=y_r - p[\lambda](\bx_r) .
        \label{e:partialG}
    \end{equation}
    
    The Proposition below characterizes  that in order to solve Problem \ref{p:main} it is sufficient to find critical points of $G$.
      {  It  is part of  the  Lagrangian duality between the primal formulation
    of Problem \ref{prob:aze}  and the dual formulations (\ref{e:GV}) or (\ref{eq:newgg}).}
    
    \begin{proposition}
        Take    $p\in \sP^n_{\mathbb{K},+}[ \mathbf X]$ and an unisolvent set of interpolation points $\left(\mathbf \bx_r\right)_{1\leq r \leq r_*}$ in $\mathbb{K}$.  Consider $y_r = p(x_r)$ for $1\leq r \leq r_*$.  The following 
        properties are equivalents
        \begin{itemize}
            \item $\lambda^*\in\cD$ is a critical point of $G$, namely $\nabla G(\lambda^*) = 0$.
            \item $\lambda^*\in\cD$ minimizes $G$.
            \item $p(\bX)= p[\lambda^*](\bX)$.
        \end{itemize}
        \label{prop:main}
    \end{proposition}
    \begin{proof}
        Since $G$ is closed convex, local minima coincide exactly with critical points, so 
        the first two  points are equivalent. The equivalence between the first and third assertions follows from \eqref{e:partialG} and the unisolvence assumption.
    \end{proof}
    
    \subsection{Number of squares}
    Let us  precise the number of squares  in the  SOS formula 
    (\ref{e:plambda}).  This information is additional with respect to Theorem \ref{t:main}. It brings the possibility to have a cheaper implementation.
    
    \begin{lemma} \label{lemma:nos}
        The number of non zero polynomials in  $\sum_{i=1}^{r_*} q_{ij}^2[\lambda](\bX)$ is less or equal to 
        $r_j$.
    \end{lemma}
    
    \begin{proof}
        By construction
        $\left(\mathbf U_1(\lambda), \dots, \mathbf U_{r_*}(\lambda) \right)=\mathbf U(\lambda)=M(\lambda)^{-1}$
        is a block diagonal matrix. The blocks have size  $r_1\times r_1$ until  $r_{j_*}\times r_{j_*}$.
        So, for a given $j$, the polynomials  $q_{ij}[\mathbf X]$ vanish for $1\leq i  \leq r_1+\dots+r_{j-1}$ and
        for  $r_1+\dots+r_{j-1}+r_{j}+1\leq i \leq r_*$.
    \end{proof}

    \begin{remark}
        The result of Lemma \ref{lemma:nos} is nevertheless non optimal 
        in dimension $d=1$. Indeed consider the Luk\'acs Theorem (see Proposition \ref{t:lukacs})
        in the odd  case  $n=2k+1$ and take $g_1(\mathbf X)=\mathbf X$ and $g_2(\mathbf X)=(1-\mathbf X)$ as in (\ref{eq:weights}). So $r_*=n+1$ and  $r_1=r_2=k+1$. Assume that there exists a critical point
        $\lambda_*$ to $G$. Then (\ref{e:plambda}) yields a representation
        $
        p(\mathbf X)= \mathbf X \sum_{i=1}^k p_{i1}^2[\lambda_*](\mathbf X)
        +(1-\mathbf X) \sum_{i=k+1}^{2k} p_{i2}^2[\lambda_*](\mathbf X)$. 
        In terms of the number of squares, here $2k$, it is clearly non optimal with respect to the result of the Luk\'acs Theorem
        which involves only two polynomials whatever $n$.
    \end{remark}

    \section{Coercivity  of  $G$ (Proof of Theorem \ref{th:th2} Item 1)}\label{s:properties}

    A sufficient condition for the existence of a critical point is that  $G$ is infinite at infinity, this is called coercivity, 
    \begin{equation} \label{eq:coco}
        \lim_{\|\lambda\|\rightarrow +\infty}G(\lambda)=+\infty.
    \end{equation}
    A sufficient condition for the uniqueness of the critical points is strict convexity.
    
    In the following, we start in Section~\ref{s:boundedness} by investigating the asymptotic behavior of $G$ along rays starting at $0$. From this knowledge we derive conditions characterizing coercivity in Section~\ref{s:coercivity}. We  characterize strict convexity in Section~\ref{s:strict}.
    
    \subsection{The asymptotic cone}\label{s:boundedness}
    
    There are two types of directions  in $\cD$. For $\mathbf d\in \RR^{r_*}$ with $\|\mathbf d\| = 1$, one  defines
    the rays $R_\mathbf d:= \left\{ \lambda=t\mathbf d\mid  t \geq 0 \right\}$ issued from   the starting point 
    $0\in R_\mathbf d$. Two possibilities occur:
    either $R_\mathbf d$ intersects the boundary 
    $\partial D$
    either it does not. In the first case  
     if one notes $t_\mathbf d>0$ the unique real number such that $t_\mathbf d\mathbf d\in \partial
    \mathcal D$, then $\lim_{t\rightarrow t_\mathbf d^-}G( t\mathbf d)=+\infty$.
    So the function $G$ is bounded from below and coercive in the direction 
     $\mathbf d$.

    In this section one is  interested in the rest of the directions. They generate the so-called asymptotic cone or recession cone of $\cD$. The asymptotic cone is closed,  independent of the starting point and is 
    classically defined \cite{hiriart_1993_convex} by
    $
    C_\infty=\left\{ 
    \lambda\in \mathbb R^{r_*}\;\text{such that}\;\forall \mu\in \cD ,\; t\geq 0 ,\;\mu + t\lambda\in\cD 
    \right\}$.

    \begin{lemma}\label{l:asympcone}
        The asymptotic cone of $\cD$ is 
        $      C_\infty=
        \left\{ \lambda \in  \mathbb R^{r_*}\mid \sum_{r= 1}^{r_*} \lambda_r B_r \blue{\succeq} 0 \right\} 
        $.
    \end{lemma}
    \begin{proof} Let $\lambda,\mu$ such that $\sum_{r= 1}^{r_*} \lambda_r B_r \succeq 0$ and $I + \sum_{r= 1}^{r_*} \mu_r B_r 
    \succ 0$. Then, $I + \sum_{r= 1}^{r_*} (\mu_r + t\lambda_r) B_r > 0$ for all $t\geq0$, so
        $\lambda$ belongs to the asymptotic cone. Conversely let $\lambda$ such that for all $\mu$ and $t\geq0$, $\mu + t\lambda\in\cD$. If $\sum_{r= 1}^{r_*} \lambda_r B_r$ had a negative eigenvalue then for $t$ large enough $I + t\sum_{r= 1}^{r_*} \lambda_r B_r$  would also have a negative eigenvalue which would contradict the fact that $t\lambda\in\cD$.
    \end{proof}
    
    The main question  is the asymptotic behavior of $G$ in directions in $C_\infty$.   
     
     Some preparatory material is provided.
    One introduces the polynomial valued vector $L(\mathbf X)$  with components being the Lagrange polynomials associated with the set of points $(\bx_r)_{1\leq r \leq r_*}$ evaluated at $\bx$, namely
    \begin{equation} \label{e:lagrange}
        L(\mathbf X)=\left(l_r(\mathbf X) \right)_{1\leq r \leq r_*} \in \RR^{r_*},
    \end{equation}
    where the Lagrange interpolation polynomials $l_r\in \sP^n[\mathbf X]$ are defined by $l_r\left(\mathbf x_s \right)=\delta_{rs}$
    for $1\leq r,s\leq r_*$, where $\delta_{rs}$ denotes the Kronecker symbol.
    The vector $L(\mathbf X)$ will be called a Lagrange vector.
    The polynomial $p$ which takes the value $y_r$ at $\mathbf x_r$ satisfies the Lagrange interpolation formula
    \begin{equation} \label{eq:plambda}
        p(\mathbf X)= \sum_{r=1}^{r_*}y_r l_r(\mathbf X)=\lla \mathbf y, L(\mathbf X)\rra.
    \end{equation}
    One can show another  interpolation property characteristics of our problem.
    
    \begin{lemma}\label{l:lagrangeCinf} One has
        $          B(\bX)=\sum_{r=1}^{r_*}l_r(\bX)B_r$. 
        For  $\bx\in\mathbb K$, $B(\bx)$ is positive semidefinite and
        $      L(\mathbf x) \in  C_\infty $.
    \end{lemma}
    
    \begin{proof}
        Let $\mathbf{W}, \mathbf{Z}\in \RR^{r_*}$ be   
        the coefficients of some polynomials $(p_{j})_{1\leq j\leq j_*}$ and  $(q_{j})_{1\leq j\leq j_*}$.
        By definition (\ref{e:defB}-\ref{e:defB2}) of $B(\bx)$ which is symmetric  one  knows that
        $$
        \lla \mathbf{W}, \left(B(\bx) - \sum_{r=1}^{r_*} l_r(\mathbf x) B_r \right) \mathbf{Z} \rra 
        =\sum_{j=1}^{j_* }\left(g_j(\mathbf x)p_{j}(\mathbf x) q_{j}(\mathbf x) - \sum_{r=1}^{r_*} l_r(\mathbf x)  g_j(\mathbf x_r) p_{j}(\mathbf x_r) q_{j}(\mathbf x_r) \right)
        =0.
        $$
        Since  $\mathbf{W},\mathbf{Z}$ are arbitrary, it yields the first part of the claim.
        Also for $\mathbf x\in \mathbb K$, one has that $g_j(\mathbf x)\geq 0$. Therefore 
        $
        \lla \mathbf{W}_1, B(\bx) \mathbf{W}_1 \rra=\sum_{j=1}^{j_* }g_j(\mathbf x)p_{j}(\mathbf x)^2\geq0$
        which yields that $B(\bx)\blue{\succeq} 0$. One gets that    $L(\mathbf x) \in  C_\infty$.
    \end{proof}
    
    In the following there are three different results concerning the behavior of $G$ in the asymptotic cone: either,  Lemma \ref{l:nonegBound},  $\inf_{t>0, \lambda \in C_\infty}G(t\lambda )=-\infty$; 
    or,   Proposition \ref{p:Farkas},  $\inf_{t>0, \lambda \in C_\infty}G(t\lambda ) >-\infty$; or even better, Proposition \ref{p:coercive},
    the function $G$ is coercive. 
    
    \begin{lemma}\label{l:nonegBound}
        Assume that there exists $\mathbf z\in \mathbb K$ such that $p(\mathbf z)<0$. Then $
        \lim_{t\rightarrow+\infty}G(tL(\mathbf z)) = -\infty$ and thus the corresponding function $G$ is not bounded from 
        below in $C_\infty$.
    \end{lemma}
    
    \begin{proof}
        The half line generated by $L(\mathbf z)$ is included in  $\cD$ by Lemma~\ref{l:lagrangeCinf} and so all for $t\geq 0$, one has
        $
        G\left(  tL(\mathbf z)\right) =
        \mbox{tr}\left(M( tL(\mathbf z))^{-1}\right) +t p(\mathbf z)
        $. Since $\lambda = tL(\mathbf z)\in C_\infty$, one has $M( \lambda )\blue{\succeq} I$ so
        $
        G(t\lambda)\leq r_*  +t p(\mathbf{z})\underset{t\to\infty}{\longrightarrow}-\infty$. 
    \end{proof}
    
    \begin{proposition} \label{p:Farkas}
        Consider   $p\in \sP^n_{\mathbb{K},+}[ \mathbf X]$, a unisolvent set of interpolation points $\left(\mathbf \bx_r\right)_{1\leq r \leq r_*}$ in $\mathbb{K}$ and define $y_r = p(\bx_r)$ for $1\leq r \leq r_*$. The following properties are equivalent.
        \begin{itemize}
            \item  For any $\lambda\in C_\infty$, one has $\lla\lambda, \mathbf y\rra \geq 0$.
            \item  There exists  polynomials $q_{ij}$ for $1\leq  j \leq j_*$ and $1\leq i \leq r_*=r_*$ such that
            \[
            p(\bX)=\sum_{j=1}^{j_*} g_j(\bX) \sum_{i=1}^{r_*} q_{ij}^2(\bX) .
            \]
        \end{itemize}
    \end{proposition}
    
    \begin{proof}
        
        For $\mathbf W\in \mathbb R^{r_*}$,
        define the vector $s_\mathbf W = \left(\lla B_r\mathbf W, \mathbf W \rra\right)_{1\leq r \leq r_*}\in \mathbb{R}^{r_*}$. 
        A    equivalent definition   of $C_\infty$  is 
        $
        C_\infty=\left\{ \lambda \in \mathbb R^{r_*} \mbox{ such that } \lla s_\mathbf W, \lambda \rra  \geq 0
        \mbox{ for all } \mathbf W\in \mathbb R^{r_*}
        \right\}$. 
        In order to prove the result, one can invoke  the Generalized Farkas Theorem  
        (\cite[Theorem III.4.3.4 page 131]{hiriart_1993_convex} with the correspondence  $\mathbf y=\mathbf b$).
        It already states that our first assertion is equivalent to $\by$  being in the closed convex conical hull of the linear forms $s_\mathbf W$, that
        is
        $   \mathbf y = \sum_{i=1}^{r_*} \alpha_i s_{\mathbf W_i}$ where $ \alpha_i\geq 0$ for all $i$,  and $r_*$ is sufficiently large.
        It is rewritten as $   \mathbf y = \sum_{i=1}^{r_*}  s_{\mathbf Z_i}$ for $\mathbf Z_i=(\alpha_i)^\frac12 \mathbf W_i$.
        Using \eqref{e:defB2}, the latter rewrites as our second assertion.
    \end{proof}
    
    \subsection{Coercivity}\label{s:coercivity}
    
    Now we investigate the conditions such that $G$ is  infinite at infinity (coercivity).
    A first negative result about coercivity is the following. The proof easily adapted from the one of Lemma~\ref{l:nonegBound}.
    
    \begin{lemma}
        Assume there exists $\mathbf z\in \mathbb K$ such that $p(\mathbf z)=0$. Then  $G(tL(\mathbf z))$ remains bounded as $t\to+ \infty$
        and $G$ is not coercive. 
        \label{l:touchzero}
    \end{lemma}

    Thus we can only hope for coercivity starting from strictly positive polynomials. Let us know define a specific useful polynomial denoted as $p_B$. 
    
    \begin{definition}\label{d:specialPoly}
        Define the polynomial
        $
        p_B(\bX)= \mathrm{tr}\left(B(\bX)\right)\in\sP^n_{\mathbb{K},+}[\bX]$, 
        where $B(\bX)$ is the matrix defined in \eqref{e:defB}.
    \end{definition}

    A key property of this polynomial is the following.
    
    \begin{lemma}\label{l:specialp}
        Assume that the matrices $\left\{B_r\right\}_{1\leq r \leq r_*}$ are linearly independent. Then there exists a constant $c_*>0$ such that 
        \begin{equation} \label{eq:235}
            c_* \| \lambda\|\leq\sum_{r=1}^{r_*}\lambda_rp_B(\bx_r), \qquad \forall\lambda\in C_\infty.
        \end{equation}
    \end{lemma}
    
    \begin{proof}
        Let $\lambda\in C_\infty$. The matrix  $\sum_r \lambda_r B_r$ is symmetric and positive semidefinite. So  its matrix norm can be controlled by its largest eigenvalue and thus by its trace, namely
        $
        \left\nrm \sum_{r=1}^{r_*} \lambda_r B_r \right\nrm\leq\mbox{tr}\left(\sum_{r=1}^{r_*}
        \lambda_r B_r\right) = \sum_{r=1}^{r_*}\lambda_r  p_B(\bx_r) $. 
        Second  we also know that $\lambda\rightarrow\sum_{r=1}^{r_*} \lambda_r B_r$ is injective thanks to the linear independence assumption. Thus there a constant $c_*>0$ such that 
        $
        c_*\| \lambda\| \leq \left\nrm  \sum_{r=1}^{r_*} \lambda_r B_r\right\nrm$. 
        Combining both inequalities ends the proof.
    \end{proof}

    { 
    \begin{remark}
      The assumption of linear independence of $\left\{B_r\right\}_{1\leq r \leq r_*}$ is close but different than  the condition of  Linear Independence Constraint Qualification (LICQ), see \cite[Section 12.2]{nocedal_2006_numerical}, which in our setting says that the matrices $\left\{B_r\mathbf{U}\right\}_{1\leq r \leq r_*}$ are linearly independent for any matrix $\mathbf{U}$ such that $\sum_{i=1}^{r_*}\left< B_r \mathbf U_i, \mathbf U_i\right> = y_r$ for $1\leq r \leq r_*$.  One may prove by contradiction that LICQ implies 
our assumption.
      \end{remark}
      }
      
          \begin{proposition}\label{p:coercive}
        Let $p\in\sP_{\mathbb{K},+}^n[\bX]$ which admits a SOS
        (\ref{e:sos}).  Take a unisolvent set of interpolation points $\left(\mathbf \bx_r\right)_{1\leq r \leq r_*}$ in $\mathbb{K}$ and assume that the corresponding matrices $\left\{B_r\right\}_{1\leq r \leq r_*}$ are linearly independent. Take  $\eps>0$ and set
        $p^\eps = p +  \eps p_B$.  
        Then the function $G^\eps$ built from $\bx_r$ and $y_r^\eps = p^\eps(\bx_r)=y_r+\eps p_B(\mathbf x_r)$ for $1\leq r\leq r_*$ is coercive.
    \end{proposition}

    \begin{proof}
        The asymptotic cone  $C_\infty$ does not depend on $\mathbf y$ or $\mathbf y^\eps$ and we desire
        to show firstly that $G^\eps$ grows linearly to infinity for directions  in $C_\infty$.
        One has the identity
        $
        \sum_{r = 1}^{r_*}\lambda_r y_r^\eps=
        \sum_{r = 1}^{r_*}\lambda_r y_ r +\varepsilon \sum_{r = 1}^{r_*}\lambda_r p_B(\bx_r) $. 
        Take $\lambda\in C_\infty$:
        proposition \ref{p:Farkas} yields  $ \sum_{r = 1}^{r_*}\lambda_r y_ r \geq 0$ because $p$ is a SOS by assumption;
        then 
        Lemma \ref{l:specialp} shows that for any $\lambda\in C_\infty$
        $
        \sum_{r = 1}^{r_*}\lambda_r y_r \geq 0 + \varepsilon c_* \|\lambda\| $ which yields uniform coercivity in the directions
        in the asymptotic cone.
        
        In order to show coercivity (\ref{eq:coco}) which is a stronger statement, the proof is by contradiction. 
        Assume it does not hold. Then there exists a constant $K\in\RR$ as well as a sequence $\left(t_m,\mathbf d_m\right)_{m\in \mathbb N}$ such that
        $t_m\rightarrow +\infty$, $\|\mathbf d_m\|=1$ and $G(t_m  \mathbf d_m)  \leq K$.
        By convexity, and since $G(0) = r_*$, one has  $G( t  \mathbf d_m)  \leq \max(r_*,K)$ for $t \in [0,  t_m]$.
        Up to the extraction of a sub-sequence there exists $\bd_*$ with $\|\mathbf d_*\|=1$, such that  $G( t  \mathbf d_*)  \leq\max(r_*,K)$  for $t \in \mathbb R^+$.
        In particular the ray with direction $\mathbf d_*$ cannot intersect the boundary $\partial\cD$ so it belongs to the  asymptotic
        cone $C_\infty$. By the first estimate $G(t\mathbf d_*)\geq  \varepsilon c_* t$, so it cannot be bounded which yields 
        the contradiction.
    \end{proof}
    

    \subsection{Strict convexity}\label{s:strict}
    
    Strict convexity, if it holds, yields uniqueness of a critical point.
    This information is additional to Item 1 of Theorem \ref{th:th2}. 
    
    \begin{proposition}\label{p:strictconvex}
        Let $p\in\sP_{\mathbb{K},+}^n[\bX]$ be strictly positive on $\mathbb{K}$. Take a unisolvent set of interpolation points $\left(\mathbf \bx_r\right)_{1\leq r \leq r_*}$ in $\mathbb{K}$ and assume that the corresponding matrices $\left\{B_r\right\}_{1\leq r \leq r_*}$ are linearly independent. Then $G$ is strictly convex over its domain $\mathcal D$.
    \end{proposition}
    
    \begin{proof}
        From \eqref{e:derivatives} 
        one has that 
        $
        \lla \nabla^2G(\lambda) \mu, \mu \rra = 2 \sum_{i=1}^{r_*}
        \lla A_i(\mu,\lambda) , M(\lambda)^{-1} A_i(\mu,\lambda)\rra \geq 0
        $ for all $\mu\in\RR^{r_*}$, 
        where $ A_i(\mu,\lambda)=(\sum_{r=1}^{r_*}\mu_r B_r) \mathbf U_i(\lambda)$
        for $ 1\leq i \leq r_*$. Since $M(\lambda)^{-1}$ is positive definite, its columns $\mathbf U_i(\lambda)$ form a basis.
        
        By contradiction, assume now $G$ is not strictly convex. There exists $\mu\neq 0$ such that
        $ \lla \nabla^2G(\lambda) \mu, \mu \rra=0$. So 
        the vectors $A_i(\mu,\lambda)$  vanish for all $i$.  So  $\sum_{r=1}^{r_*}\mu_r B_r = 0$, and  $\mu = 0$ by linear independence of the matrices $(B_r)_{r = 1,\dots, r_*}$. This is a contradiction so
        $\nabla^2G(\lambda)>0$ and 
        $G$ is strictly convex.
    \end{proof}

    The strict convexity of $G$ can be measured with the minimal eigenvalue of its Hessian
    $
    \alpha(\lambda)= \inf_{\mu\neq 0}  \frac{\lla \nabla^2G_\bV(\lambda) \mu, \mu \rra}{\|\mu\|^2} > 0$, 
    for any $\lambda\in\cD$. 
    An important property which motivates the design of one of our numerical methods is the following.
    
    \begin{lemma}\label{l:cubic}
        Under the assumptions of Proposition~\ref{p:strictconvex}, then $\alpha$ has a cubic degeneracy  at infinity in the interior of the asymptotic cone of $\cD$.  For all $\bd\in\RR^{r_*}$ such that $\|\bd\| = 1$ and $\sum_{r=1}^{r_*}d_rB_r  \blue{\succ} 0$, there is $C_\bd>0$ such that
        $
        \alpha(t\bd) \leq C_\bd (1 + t)^{-3} $  for all $t\geq 0$.
    \end{lemma}
    
    \begin{proof}
        Let $\lambda = t\bd$. For a constant $C$ depending only on the data, one has
        $
        \lla \nabla^2G(\lambda) \mu, \mu \rra\leq C \| M(\lambda)^{-1}\| ^3\|\mu\|^2 $. 
        Under the assumptions the minimal eigenvalue of $M(\lambda)$ is given by $1+e_{\bd}t$ with  $e_\bd$ the minimal eigenvalue of $\sum_{r=1}^{r_*}d_rB_r$. Hence $\| M(\lambda)^{-1}\|=O((1+t)^{-1})$.
        
    \end{proof}

    \section{Univariate polynomials on a segment (Theorem \ref{th:th2} Item 2)}\label{s:univariate}
    
    In this section, we focus on univariate polynomials, namely when $d=1$, over the segment $\mathbb{K}=[0,1]$. 
    This case is interesting because 
    it is central for
    for numerical computation of functions of one variable and also 
    one can easily prove the coercivity and the strict convexity.
    The notation is simplified by using the real variable $x\in \mathbb R$, more adapted to  analytical methods.

    We check that the various assumptions granting coercivity and strict convexity are satisfied. In view of Proposition~\ref{p:Farkas}, Proposition~\ref{p:coercive} and Proposition~\ref{p:strictconvex} of the previous section, it suffices to exhibit  an appropriate choice of functions $(g_j)_j$  and of interpolation points
    such that: any non-negative polynomial admits a (possibly non-explicit) SOS decomposition; and 
    the matrices $\{B_r\}_r$ are linearly independent.
    The first point follows from the \emph{Markov-Luk\'acs} Theorem,  see \cite{szego_1975_orthogonal, despres_2017_polynomials, despres_2017_erratum,krein_1977_markov} for a proof. 
    
    \begin{proposition}[\emph{Markov-Luk\'acs}] Let us consider $p\in \sP^{n}[x]$ and $\mathbb K=[0,1]$. 
        \begin{itemize}
            \item\textbf{Even case:} If $n = 2k$, then $p$ is non-negative on $\mathbb K$ if and only if there are polynomials $a$ and $b$ with degree less or equal to $k$ and $k-1$ respectively such that
            \begin{equation} \label{eq:lukacseven}
                p(x)=a^2(x) + x(1-x) b^2(x).
            \end{equation}
            \item\textbf{Odd case:} If $n = 2k+1$, then $p$ is non-negative on $\mathbb K$ if and only if there are polynomials $a$ and $b$ with degree less or equal to $k$ such that
            \begin{equation} \label{eq:lukacsodd}
                p(x) = x a^2(x)+ (1-x)b^2(x). 
            \end{equation}
        \end{itemize}
        \label{t:lukacs}
    \end{proposition}
    
    Now let us precise the setting. 
    One takes  $j_* = 2$ and 
    \begin{equation}\label{eq:weights}
        \left\{
        \begin{array}{llll}
            \text{for }n\text{ is even}: &
            g_1(x) = 1 &\text{and}& g_2(x) = x(1-x), \\[.75em]
            \text{for }n\text{ is odd}: &
            g_1(x) = x &\text{and}& g_2(x) = 1-x.
        \end{array}
        \right.
    \end{equation}
    Concerning the interpolation points, we choose any $r_* = n+1$ distinct points $(x_r)_{r = 1,\dots,n+1}$ on the segment $[0, 1]$.
    The polynomials are represented along monomials so that the matrices $B_r$ have the block structure
    \begin{equation} \label{eq:Br1D}
        B_r=\left(
        \begin{array}{cc}
            g_1(x_r) \mathbf w^r_{1} \otimes   \mathbf w^r_{1} & 0 \\
            0 & g_2(x_r) \mathbf w^r_{2} \otimes   \mathbf w^r_{2}
        \end{array}
        \right)\in \mathbb R^{(n+1)\times (n+1)}
    \end{equation}
    where 
    \[
    \left\{
    \begin{array}{ll}
        \mbox{for }n=2k:&
        \mathbf w_{1}^r =  \left(1, x_r, \dots, x_r^k\right)^t \mbox{ and }
        \mathbf w_{2}^r =  \left(1, x_r, \dots, x_r^{k-1}\right)^t ,
        \\[.75em]
        \mbox{for }   n = 2k+1: &
        \mathbf w_{1}^r =  \mathbf w_{2}^r = \left(1, x_r, \dots, x_r^k\right)^t.
    \end{array}
    \right.
    \]
    With these notations, the equalities  \eqref{eq:lukacseven} and \eqref{eq:lukacsodd} are equivalent to
    $ y_r = \left< B_r \bU, \bU \right> $ for $1\leq r \leq n+1$. In the odd case $n =2k+1$ one has
    $
    \bU=(a_0, \dots, a_k, b_0, \dots b_k)^t\in \mathbb R^{n+1}
    $ with $a(x) = \sum_{l=0}^ka_lx^l$ and $b(x) = \sum_{l=0}^kb_lx^l$.
    In the even case $n=2k$, $\bU=(a_0, \dots, a_k, b_0, \dots b_{k-1})^t\in \mathbb R^{n+1}$.

    \begin{corollary}[of Proposition~\ref{p:Farkas}]\label{cor:Farkas}
        Take  $p\in P_{[0,1],+}^n$ and  set $y_r = p(x_r)$. Then, for all $\lambda\in C_\infty$, one has that 
        $ \lla\lambda, \mathbf y\rra \geq 0$.
    \end{corollary}
    
    \begin{proof}
        Indeed the second statement of Proposition \ref{p:Farkas} holds with $i_* = 1$ by taking $p_{11} = a$ and  $p_{12} = b$ with $a,b$ provided by Proposition \ref{t:lukacs}.
    \end{proof}
    
    Let $\lambda\in\RR^{n+1}$. Using the structure (\ref{eq:Br1D})  of the matrices $B_r$, one has
    the Hankel matrices
    \begin{equation}\label{e:hankel}\sum_{r=1}^{n+1}\lambda_r B_r = \left(\begin{array}{cc}
            H_1 & 0 \\
            0 & H_2
        \end{array}
        \right)
    \end{equation}
    where
    \[
    \left\{
    \begin{array}{lll}
        \text{for } n = 2k: &  \ds\lla H_1\mathbf{v}, \mathbf{w} \rra = \sum_{i,j = 0}^{k} s_{i+j+1} v_i w_j ,&  \ds\lla H_2\mathbf{v}, \mathbf{w} \rra = \sum_{i,j = 0}^{k} (s_{i+j} - s_{i+j+1}) v_i w_j,\\[1em]
        \text{for } n = 2k+1 : &
        \ds\lla H_1\mathbf{v}, \mathbf{w} \rra = \sum_{i,j = 0}^{k} s_{i+j+1} v_i w_j ,&  \ds\lla H_2\mathbf{v}, \mathbf{w} \rra = \sum_{i,j = 0}^{k-1} (s_{i+j+1} - s_{i+j+2}) v_i w_j.
    \end{array}
    \right.
    \]
    The $s_i$'s are given by  
    $	s_i = \sum_{r=1}^{n+1}\lambda_r x_r^i$. 
    The linear map $\lambda\mapsto(s_0, \dots,s_n)$ is one to one, since
    $(s_0, \dots,s_n)$ is obtained by multiplying $\lambda$ by a Vandermonde matrix, which is invertible. 
    A direct consequence is the following.
    
    \begin{lemma} \label{lemma:ind}
        The matrices $\left\{B_r\right\}_{1\leq r \leq r_*}$  are linearly independent.
    \end{lemma}
    
    \begin{proof}
        Assume  $\sum_{r=0}^n \lambda_r B_r = 0$. Then (\ref{e:hankel}) and the definition of $H_1$ and $H_2$ yields that $s_0=\dots =s_n=0$.   It yields  $\lambda = 0$. So the  $\left\{B_r\right\}_{1\leq r \leq r_*}$  are linearly independent.
    \end{proof}
    
    \begin{proposition}\label{t:main1D}
        For any univariate polynomial $p$ that is strictly positive on $\mathbb{K} = [0,1]$, the associated function $G$ is strictly convex and coercive. As a consequence, it has a unique critical point $\lambda^*\in \mathcal D$
         which defines a sum of squares decomposition $p[\lambda^*] = p$.
    \end{proposition}
    \begin{proof}
        Thanks to Corollary~\ref{cor:Farkas} and Lemma~\ref{lemma:ind}, the assumptions of Proposition~\ref{p:coercive} and Proposition~\ref{p:strictconvex} are satisfied which yields the result.
    \end{proof}

    \section{Numerical algorithms}\label{s:methods}

    \blue{The numerical methods are based on the minimization of the dual function $G$ either by a descent type algorithm, either by the direct search of a critical point with a Newton type methods. Let us emphasize that as $G$ is a proper strictly convex and coercive function on its domain, its minimization is equivalent to the search of a critical point.} All the methods enter the generic iterative framework
    \begin{equation}
        \lambda^{m+1} = \lambda^{m} - \tau_m H_m^{-1} \nabla G(\lambda^m) ,\qquad\lambda^0 = 0 ,
        \label{e:iterativemeth}
    \end{equation}
    with $H_m$ and $\tau_m$ to be defined.
    {In terms of complexity, the cost of one iteration is essentially the cost of computation $\nabla G(\lambda^m)$ with  formula \eqref{e:derivatives}. Indeed, one needs to compute the inverse of $M^{-1}(\lambda)$ and then do $O(r*)$ matrix multiplications with the matrices $B_r$. It yields a cost in $O(r_*^3) + O(r_*\times r_*^3)$. Observe that the Hessian of $G$ is not more expensive to compute since all the vectors $B_r{\bf U}_i$ and $M(\lambda)^{-1}$ have already been computed. Therefore, one needs to do $O(r*)$ matrix multiplications to get the $M(\lambda)^{-1}B_r{\bf U}_i$ and from these matrices the assembling of the Hessian is no more than $O(r_*^4)$. In the end, the cost of the inversion of $H_m$, whatever how it is defined, is less than the evaluation of the gradient.}

    \subsection{Choices for $H_m$} Let us first explain the various choice for the matrix $H_m$.
    \subsubsection{Forward descent method}\label{s:forward}
    The first method we use is the classical descent method which consists in taking  
    $
    H_m = I  \mbox{  the identity matrix}$. 
    
    \subsubsection{Backward descent method}\label{s:eulerimp}
    
    Given a sequence  of positive time steps $\tau_m$, the following iterative scheme 
    $
    \tilde{\lambda}^{m+1} = \mathrm{arg}\min_{\lambda\in\cD} G(\lambda)  +\frac{1}{2 \tau_m}\|\lambda -  \tilde{\lambda}^m\|^2 $ with initial guess
    $ \tilde{\lambda}^0 = 0$
    is well defined since $G$ is convex. It corresponds exactly to the implicit Euler discretization of the gradient flow with variable time steps. At step $m$ we look for the critical point of the strictly convex objective function by making one step of a Newton method starting at $\lambda^m$, yielding the scheme \eqref{e:iterativemeth} with
    $
    H_m = I + \tau_m \nabla^2G(\lambda_m) $.
    
    \subsubsection{Newton-Raphson method}\label{s:Newt}
    
    A  straightforward method for a direct search of the critical point of $G$ is the classical Newton method
    $    H_m = \nabla^2G(\lambda_m) $, 
    with $\nabla^2 G$ the Hessian of $G$.

    \subsubsection{Modified Newton-Raphson method}\label{s:quasiNewt}
    
    The Hessian of $G$ degenerates at infinity  as showed in Lemma \ref{l:cubic}. 
    In practice, a classical Newton-Raphson method for solving $\nabla G(\lambda) = 0$ can be inaccurate at the first iterations in some cases. Instead one may notice that $\lambda_*$ is a critical point of $G(\lambda)$ if and only if it is a critical point of $(G(\lambda) - C)^2$ where $C$ is a constant which is smaller than the infimum of $G$. One expects the latter function to grow quadratically at infinity thus improving the conditioning of the Hessian. This suggests the modified Newton method \eqref{e:iterativemeth} with
    $    H_m = \alpha_m \nabla G(\lambda_m)\otimes\nabla G(\lambda_m)  + \nabla^2G(\lambda_m) $. 
    Several choices are possible for $\alpha_m$. Following the heuristic one could impose $\alpha_m = (G(\lambda_m)-C)^{-1}$ but $C$ is not known \emph{a priori}. In practice, we found out that the empirical choice 
    $
    \alpha_m = \|\nabla  G(\lambda_m)\|/(\|\nabla G(\lambda_m)\| + \|\nabla G(0)\|)
    $ 
    yields good results. \blue{This choice is motivated by the fact that close to the critical point, the method degenerates back to the classical Newton-Raphson method.}

    \blue{
    \subsection{Choice of the time step $\tau_m$} Now we detail the choice of adaptive time step. A preliminary concern is whether one can ensure that every iterate stays in the domain of $G$. 
    
    \subsubsection{Maximal time step}\label{s:max}
    It is possible to guarantee  the condition $\lambda_{m}\in \mathcal D$ for any $m$ by imposing a simple threshold on the time step. Indeed start from $\lambda_m\in \mathcal D$.  Since $\lambda^{m+1}=\lambda^{m}-\tau_m \mathbf d^m$ for a given direction
    $\mathbf d^m=(d_r^m)$, the condition $\lambda^{m+1}\in \mathcal D$ is satisfied
    provided $M(\lambda^m)- \tau_m \sum _r d_r^m B_r\succeq 0$. A sufficient condition is that $\tau_m\leq\tau_{\max}$ with $\tau_{\max}$ such that
    $
     \mu_{\max}\left(  \sum _r d_r^n B_r \right) \tau_{\max}  \leq  \mu_{\min}(M(\lambda^m))
    $, 
    where $\mu_{\max}(A)$ and $\mu_{\min}(A)$ denote respectively  the maximum and minimum of the absolute values of the eigenvalue of a real symmetric square matrix $A$ (\emph{i.e.} the spectral radius and, if $A$ is positive definite, the spectral gap). This condition is very much like a CFL stability condition. In various test cases we observed that $\tau_{\max}$ is of the order of $1$ initially and tends to increase as iterates get closer to the solution.
    
    \subsubsection{Empirical adaptive time step}\label{sec:adaptive}
     The first choice of time step relies on a criteria of decay of $\|\nabla G(\lambda_m)\|$. In the case of descent methods, it differs from the more usual Wolfe condition \cite{nocedal_2006_numerical} which enforces a decay of $G(\lambda_m)$ and it is fairly close to the so-called strong Wolfe condition. We make this choice because it is well adapted to our particular setting. Indeed we recall that $\|\nabla G(\lambda_m)\|$ actually measures the Euclidean norm between the current sum of squares $(p[\lambda^m](x_r))_r$ and $\by$. By equivalence of norms and unisolvance one has $\| p -p[\lambda^m]\|_{P^n[\mathbf{X}]}\leq c_n \|\nabla G (\lambda^m)\|$, for some constant $c_n>0$ depending only on $n$, whatever the choice of norm of the space of polynomials. It is thus the natural measure of the error which has to be decreased by the iterative algorithm.
    
      The adaptive time step $\tau_m$ is defined as follows. We choose \emph{a priori} $0<\tau_{\min}\leq\tau_0\leq\tau_{\max}$. Then we define $\lambda_{m}^{(k)} = \lambda_{n} - 2^{-k} \tau_m^{(k)} H_m^{-1} \nabla G(\lambda_m)$ and denote by $k_{m}$ the smallest integer such that $\|\nabla G(\lambda_m^{(k)})\| < \|\nabla G(\lambda_{m+1})\|$. From there we define $\tau_{m+1}=\max(2^{-k_m}\tau_m,\tau_{\min}) \mbox{ for }k_m>0 $ and $ 
    \tau_{m+1}=  \min(2\tau_m,\tau_{\max}) \mbox{ for } k_m=0$.

    In the case of Newton methods, we take $\tau_0=\tau_{\max}=1$ in order to achieve the expected quadratic rate of convergence. As we shall see in the numerical results, the decrease of the time step in Newton methods coincides with a bad conditioning of the Hessian matrix. 
    }
    \blue{
    \subsubsection{Barzilai and Borwein time step}\label{s:bb}
    In the case of the forward descent method of Section~\ref{s:forward}, there is a particular choice of time step relying on the two previous iterates due to Barzilai and Borwein in their seminal paper \cite{barzilai_1988_two} (see  \cite{barzilai_1997_raydan} for an improvement of the original method which ensures global convergence, two other
    recent works are \cite{sera,gao}). It can be seen as an intermediate between 
    the classical gradient descent method of Cauchy and the Newton method as it generalizes the secant method in higher dimensions. The corresponding time step is given, for $m\geq 1$, by either
    \begin{equation}\label{eq:bb1}
     \tau_m\ =\ \frac{\lla\nabla G(\lambda_m)-\nabla G(\lambda_{m-1}),\lambda_m-\lambda_{m-1}\rra}{\|\nabla G(\lambda_m)-\nabla G(\lambda_{m-1})\|^2}\,,
    \end{equation}
    or
    \begin{equation}\label{eq:bb2}
     \tau_m\ =\ \frac{\|\lambda_m-\lambda_{m-1}\|^2}{\lla\nabla G(\lambda_m)-\nabla G(\lambda_{m-1}),\lambda_m-\lambda_{m-1}\rra}\,.
    \end{equation} 
    It is known that this method does not yield a monotone decay of either $G$ or $\|\nabla G\|$ in general. In our case it may be that $\lambda_{m+1}\notin \mathcal{D}$ with the choices \eqref{eq:bb1} and \eqref{eq:bb2}. Thus the methods are stabilized by replacing $\tau_m$ with $\tau_{\max}$ given in Section~\ref{s:max} whenever $\tau_m>\tau_{\max}$.
    }
    
    \section{Numerical experiments}\label{s:numerics}
    
    In this section, we perform various numerical experiments in order
    to illustrate the theoretical results 
    and to explore the behavior of the numerical algorithms.
    The implementation has been performed with Matlab and Python, with no noticeable difficulties.
    The maximal time step of Section \ref{s:max} is calculated
    with built-in subroutines, the extra-cost is negligible.
     \blue{In the following, we denote by ``Gradient descent'', ``BB1''and ``BB2'' the methods where $H_m = I$ and the time step is taken as in Section~\ref{sec:adaptive} and Section~\ref{s:bb}  with \eqref{eq:bb1} and \eqref{eq:bb2} respectively. The methods ``Implicit Euler'', ``Newton'' and ``Modified Newton'' correspond respectively to the choices of Section~\ref{s:eulerimp},  Section~\ref{s:Newt} and Section~\ref{s:quasiNewt} and the time step is taken as in Section~\ref{sec:adaptive}} 
    
    \subsection{Univariate polynomials on a segment}
    
    Here we consider   univariate SOS polynomials. We proceed as explained in Section~\ref{s:univariate}, except that the monomial basis is replaced here by the orthogonal basis of shifted Chebychev polynomials $(T_i(x))_{i = 1,\dots,k}$ satisfying
    $
    T_i(\cos(\theta)+1)/2) = \cos(i\theta) ,
    $
    for all $\theta\in\RR$. The only modification of the method presented earlier concerns the definition of the $D_r$ matrices which become
    $
    D_r=\mathbf w_r^t \mathbf w_r \in
    \mathbb R^{r_k\times r_k}$ with $\mathbf w_r= \left(T_0(x_r), T_1(x_r), \dots, T_k(x_r)
    \right)^t\in \mathbb R^{r_k}
    $.
    The reason is that  shifted Chebychev polynomials have much better behavior in terms of numerical approximation, since they produce 
    "uniformly distributed" polynomials in $[0,1]$, see \cite{devore} for  comprehensive mathematical treatment. On the opposite,  monomials $x^i$ which concentrate at $x=1$ for $i\rightarrow
    +\infty$ are non optimal for numerical approximation in the segment $[0,1]$. One can refer to  \cite{despres_2017_polynomials} for a comparison between the
    use of Chebychev polynomials and monomials. 
    In the following we propose different test cases to illustrate the properties of the various descent and Newton-Raphson type methods proposed in Section~\ref{s:methods}.
    For univariate polynomials,  the tests 1-2-3 are performed with the odd order option (\ref{eq:weights}) of the weights: similar results are observed with $g_1(x)=1$ and $g_2(x)=x(1-x)$, and so are not reported.
    Test 5 is performed with both the odd and even options.

    \subsubsection{Test case 1} 
    
    \blue{
    \begin{figure}
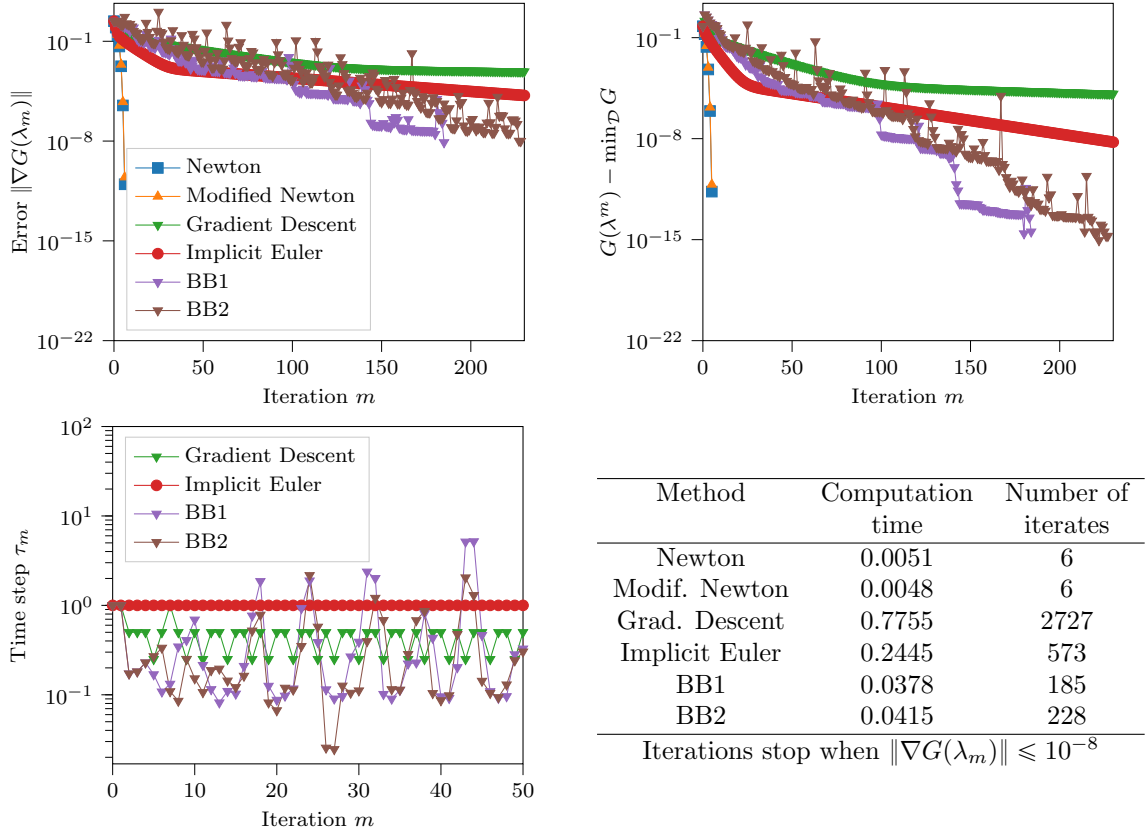

        \centering
        \footnotesize
        \begin{tabular}{cc}
            \begin{minipage}{.47\textwidth}
            \input{1d_test_1_error_new.tex}
            \end{minipage}
            &
            \begin{minipage}{.47\textwidth}
            \input{1d_test_1_G_new.tex}
            \end{minipage}\\
            \begin{minipage}{.47\textwidth}
            \input{1d_test_1_step_new.tex}
            \end{minipage}
            &
            \begin{minipage}{.47\textwidth}
            \centering
            \normalsize
            \blue{
            \begin{tabular}{ccc}\hline
            Method&Computation&Number of\\
            &time& iterates\\\hline
            Newton&0.0051&6\\
            Modif. Newton&0.0048&6\\
            Grad. Descent&0.7755&2727\\
            Implicit Euler&0.2445&573\\
            BB1&0.0378&185\\
            BB2&0.0415&228\\\hline
            \end{tabular}
            Iterations stop when $\|\nabla G(\lambda_m)\|\leq 10^{-8}$}
            \end{minipage}
        \end{tabular}
        \caption{\blue{\normalsize\label{fig:test1}{Test case 1.} Sum of square interpolation of $p(x) = x^{5} + 1$. (Top left) Error $\|\nabla G(\lambda_m)\|$ vs. iteration $m$; (Top right) Objective function $G(\lambda_m)$ vs. iteration $m$; (Bottom left) Step size $\tau_m$ vs. iteration $m$; (Bottom right) Number of iterates and computation time to reach $\|\nabla G(\lambda)\| < 10^{-8}$ for each method.}
        }
    \end{figure}
    }
    
    We compare the convergence of the methods for an easy objective polynomial, that is a polynomial with low degree and far above $0$: we take $n=5$, $r_*=i_*=n+1=6$,  $p(x) = x^{5} + 1$ and the weights $g_1(x)=x$ with $g_2(x)=1-x$ (so $j_*=2$).
    \newline 
    We observe on Figure~\ref{fig:test1} that the Newton type methods both reach the threshold precision of $10^{-8}$ after only $6$ iterations. The implicit Euler and gradient descent methods need respectively $573$ and $2727$ iterations to reach the same error:
    this low convergence  has been observed for many other test cases. This is why we continue the tests with the Newton \blue{and Barzilai and Borwein} methods only.

    \subsubsection{Test case 2}
    \blue{
    \begin{figure}
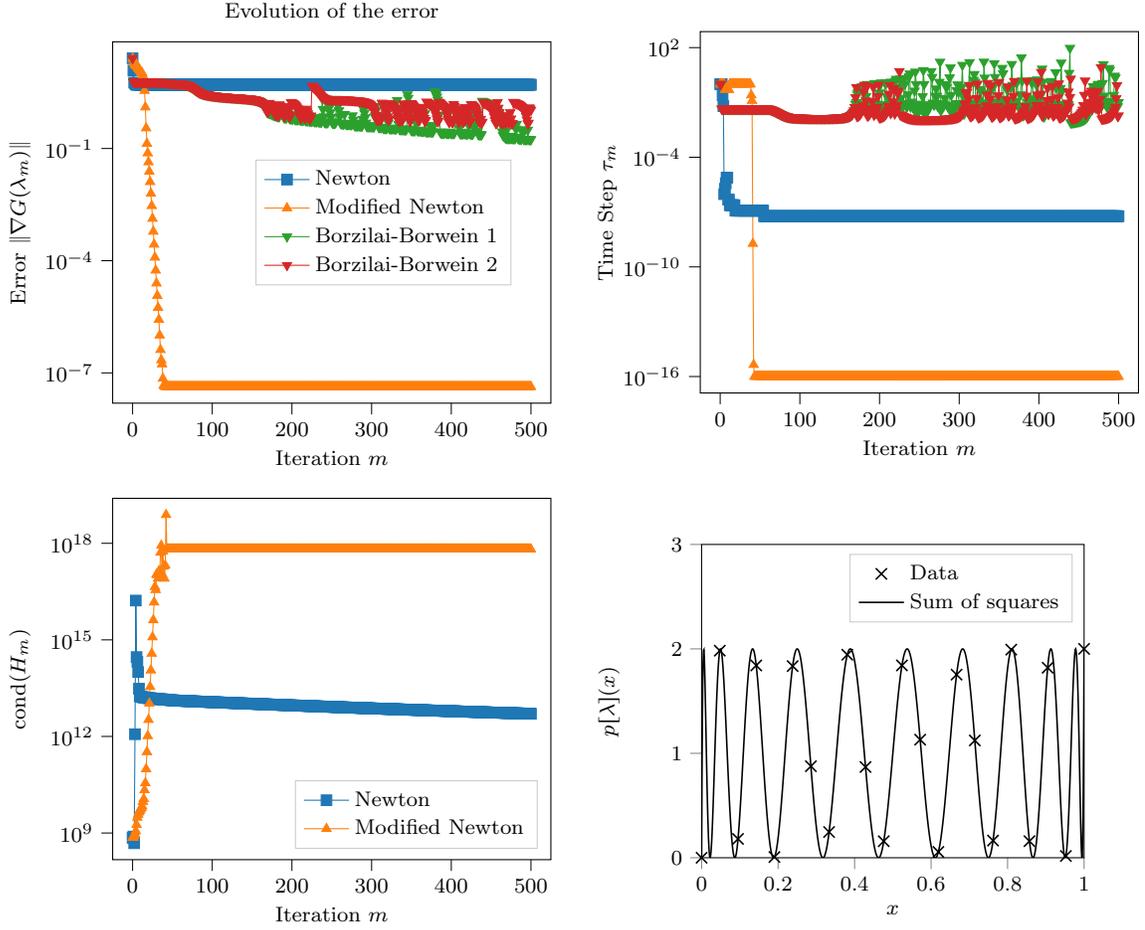

        \centering
        \footnotesize
        \begin{tabular}{cc}
            \begin{minipage}{.47\textwidth}
            \input{1d_test_2_error_new.tex}
            \end{minipage}
            &
            \begin{minipage}{.47\textwidth}
            \input{1d_test_2_step_new.tex}
            \end{minipage}\\[1em]
            \begin{minipage}{.47\textwidth}
            \input{1d_test_2_cond_new.tex}
            \end{minipage}
            &
            \begin{minipage}{.47\textwidth}
            \input{1d_test_2_data_sos.tex}  
            \end{minipage}
        \end{tabular}
        \caption{\blue{\label{fig:test2}{Test case 2.} Sum of square interpolation of $p(x) = T_{21}(x) + 1$. (Top left) error $\|\nabla G(\lambda_m)\|_2$ vs. iteration $m$; (Top right) Step size $\tau_m$ vs. iteration $m$; (Bottom left) Condition number of $H_m$ vs. iteration $m$; (Bottom right) Data $(y_r = p(x_r))_r$ and sum of squares 
            $p[\lambda](x)$ satisfying $\|\nabla G(\lambda)\| < 10^{-6}$.}}
    \end{figure}
    }
    
    In this second test case, we illustrate the better performance of the modified Newton-Raphson method compared to the other methods. We choose a highly oscillating objective polynomial with lower bound equal to $0$. It is given by $n=21$, $r_*=i_*=n+1=22$,  $p(x) = T_{21}(x) + 1$ and the weights $g_1(x)=x$ with $g_2(x)=1-x$ (so $j_*=2$).

    We observe on Figure~\ref{fig:test2} that the modified Newton-Raphson method reaches a precision of around $10^{-8}$ in $40$ iterations. In the case of the standard Newton-Raphson method, the adaptive time step quickly reduces to a very small value in order to keep decreasing the error at each iteration. A similar phenomena happens near convergence for the modified Newton-Raphson method. These behaviors can be interpreted thanks to the evolution of the condition number of the matrix $H_m$ also showed on Figure~\ref{fig:test2}. Let us recall that this matrix needs to be inverted at each iteration. On the first hand, for the Newton-Raphson method,  $H_m$ is the Hessian of $G$ which degenerates when $\lambda$ is far from the minimizer of $G$, as explained in Lemma~\ref{l:cubic}. The modified Newton-Raphson method seems to prevent a bad condition number of the tweaked Hessian in the first few iterations. On the second hand, since the objective polynomial has a  $0$ lower bound, strict convexity and coercivity of $G$ are not granted and it may explain the bad conditioning of $H_m$ near convergence in the case of the modified Newton-Raphson method. Indeed recall that when $\nabla G(\lambda_m)$ is small $H_m$ almost coincides with the Hessian in the modified Newton-Raphson method.

    \blue{Concerning the Barzilai and Borwein methods, we found out that the convergence is very slow on this test case. A minimal error $\|\nabla G(\lambda_m)\|$ of around $10^{-4}$ is attained after $100000$ iterations for the ``BB1'' method, and worse performances are obtained with the ``BB2'' method. While these methods are well-suited for many nonlinear programming problems, it seems that despite the cost of the Hessian inversion, it is significantly cheaper in terms of computational effort to use Newton type iterations on our particular problem.}
    
    Eventually we found in many numerical  experiments that, on the particular problem addressed in this paper, the modified Newton method is by far the most robust and efficient method among the ones we tested. This the reason why  we only use the modified Newton-Raphson method in the following series of tests.
    
    \subsubsection{Test case 3} 
    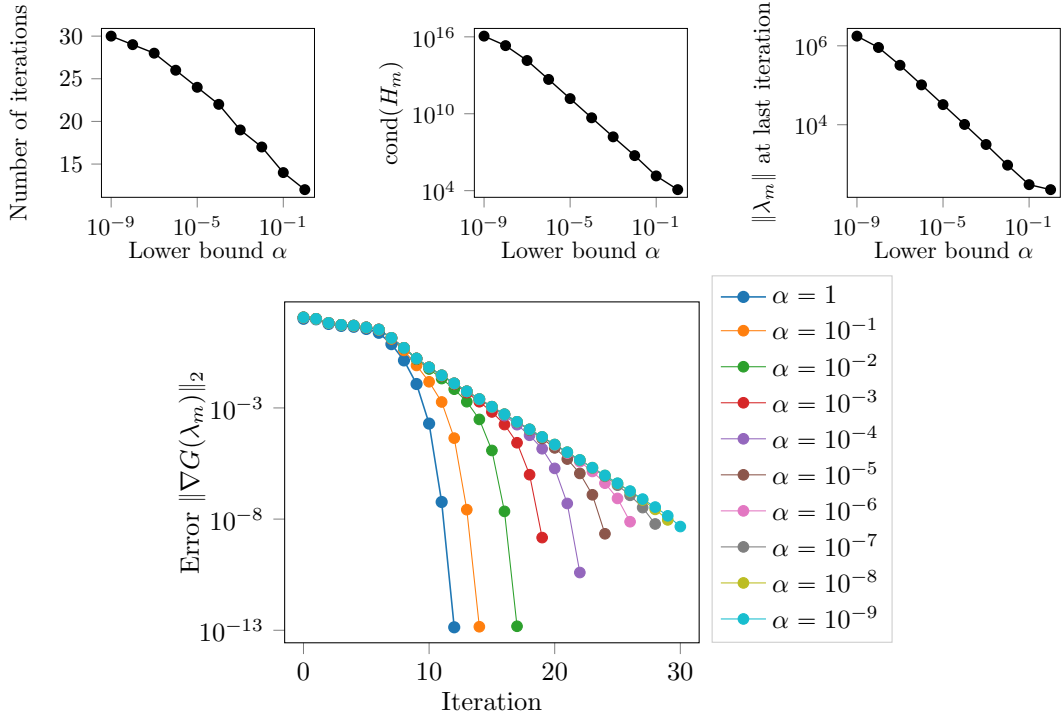
\begin{figure}
        \centering
        \begin{tabular}{c}
            \begin{tabular}{ccc}
                \scalebox{.9}{ 
\begin{tikzpicture}

\begin{axis}[
tick align=outside,
tick pos=left,
x grid style={white!69.01960784313725!black},
xlabel={Lower bound $\alpha$},
xmin=3.54813389233578e-10, xmax=2.81838293126445,
xmode=log,
y grid style={white!69.01960784313725!black},
ylabel={Number of iterations},
ymin=11.1, ymax=30.9,
width = .3\textwidth
]
\addplot [semithick, black, mark=*, mark size=2, mark options={solid}, forget plot]
table [row sep=\\]{%
1	12 \\
0.1	14 \\
0.01	17 \\
0.001	19 \\
0.0001	22 \\
1e-05	24 \\
1e-06	26 \\
1e-07	28 \\
1e-08	29 \\
1e-09	30 \\
};
\end{axis}

\end{tikzpicture} }
                &
                \scalebox{.9}{   
\begin{tikzpicture}

\begin{axis}[
tick align=outside,
tick pos=left,
x grid style={white!69.01960784313725!black},
xlabel={Lower bound $\alpha$},
xmin=3.54813389233578e-10, xmax=2.81838293126445,
xmode=log,
y grid style={white!69.01960784313725!black},
ylabel={$\mathrm{cond}(H_m)$},
ymin=3017.78659330827, ymax=4.5252564161319e+16,
ymode=log,
width = .3\textwidth
]
\addplot [semithick, black, mark=*, mark size=2, mark options={solid}, forget plot]
table [row sep=\\]{%
1	11983.8276710506 \\
0.1	140971.321494832 \\
0.01	5414985.08858649 \\
0.001	157675794.10517 \\
0.0001	4843089091.80387 \\
1e-05	151668671050.901 \\
1e-06	4763388625250.06 \\
1e-07	142601564903223 \\
1e-08	2.06424878147129e+15 \\
1e-09	1.13955728659839e+16 \\
};
\end{axis}

\end{tikzpicture} }
                &
                \scalebox{.9}{    
\begin{tikzpicture}

\begin{axis}[
tick align=outside,
tick pos=left,
x grid style={white!69.01960784313725!black},
xlabel={Lower bound $\alpha$},
xmin=3.54813389233578e-10, xmax=2.81838293126445,
xmode=log,
y grid style={white!69.01960784313725!black},
ylabel={$\|\lambda_m\|$ at last iteration},
ymin=144.789408826456, ymax=2755296.92767952,
ymode=log,
width = .3\textwidth
]
\addplot [semithick, black, mark=*, mark size=2, mark options={solid}, forget plot]
table [row sep=\\]{%
1	226.598178076442 \\
0.1	303.27612587429 \\
0.01	945.277277096854 \\
0.001	3155.97601322805 \\
0.0001	10171.1273029735 \\
1e-05	32358.7371983485 \\
1e-06	102436.906940006 \\
1e-07	321405.643918383 \\
1e-08	906636.092557093 \\
1e-09	1760551.7250253 \\
};
\end{axis}

\end{tikzpicture} }
            \end{tabular}\\[1em]
\begin{tikzpicture}

\definecolor{color0}{rgb}{0.12156862745098,0.466666666666667,0.705882352941177}
\definecolor{color1}{rgb}{1,0.498039215686275,0.0549019607843137}
\definecolor{color2}{rgb}{0.172549019607843,0.627450980392157,0.172549019607843}
\definecolor{color3}{rgb}{0.83921568627451,0.152941176470588,0.156862745098039}
\definecolor{color4}{rgb}{0.580392156862745,0.403921568627451,0.741176470588235}
\definecolor{color5}{rgb}{0.549019607843137,0.337254901960784,0.294117647058824}
\definecolor{color6}{rgb}{0.890196078431372,0.466666666666667,0.76078431372549}
\definecolor{color7}{rgb}{0.737254901960784,0.741176470588235,0.133333333333333}
\definecolor{color8}{rgb}{0.0901960784313725,0.745098039215686,0.811764705882353}

\begin{axis}[
legend cell align={left},
legend entries={{$\alpha = 1$},{$\alpha = 10^{-1}$},{$\alpha = 10^{-2}$},{$\alpha = 10^{-3}$},{$\alpha = 10^{-4}$},{$\alpha = 10^{-5}$},{$\alpha = 10^{-6}$},{$\alpha = 10^{-7}$},{$\alpha = 10^{-8}$},{$\alpha = 10^{-9}$}},
legend style={at={(1.03, 0)}, anchor=south west, draw=white!80.0!black},
tick align=outside,
tick pos=left,
x grid style={white!69.01960784313725!black},
xlabel={Iteration},
xmin=-1.5, xmax=31.5,
y grid style={white!69.01960784313725!black},
ylabel={Error $\|\nabla G(\lambda_m)\|_2$},
ymin=2.73433487057978e-14, ymax=58.8076811469988,
ymode=log,
width = .45\textwidth
]
\addplot [semithick, color0, mark=*, mark size=2, mark options={solid}]
table [row sep=\\]{%
0	10.1949229289356 \\
1	9.67463285201791 \\
2	5.98833587419192 \\
3	4.87737386039389 \\
4	4.45432471191695 \\
5	3.55808092039398 \\
6	2.39422908563428 \\
7	0.739382842266761 \\
8	0.137236420696068 \\
9	0.0119146467391738 \\
10	0.000197376640930422 \\
11	5.88564527664591e-08 \\
12	1.36077794740119e-13 \\
};
\addplot [color1, mark=*, mark size=2, mark options={solid}]
table [row sep=\\]{%
0	11.6531509038885 \\
1	10.084558617598 \\
2	6.5500421253589 \\
3	5.36999840759554 \\
4	4.96228969054074 \\
5	4.09650825383296 \\
6	3.20029597162939 \\
7	1.26859741130057 \\
8	0.38640705944634 \\
9	0.081563125925774 \\
10	0.015003449849353 \\
11	0.00185504737134548 \\
12	4.38308455736975e-05 \\
13	2.65698355255212e-08 \\
14	1.44119089311631e-13 \\
};
\addplot [color2, mark=*, mark size=2, mark options={solid}]
table [row sep=\\]{%
0	11.8003886899959 \\
1	10.1368206361775 \\
2	6.61870747006208 \\
3	5.43269981195105 \\
4	5.03055567693833 \\
5	4.17688021171034 \\
6	3.3369824462664 \\
7	1.39174899024913 \\
8	0.489420933170265 \\
9	0.155594710272631 \\
10	0.0563300581263342 \\
11	0.0209195805792124 \\
12	0.00694873544182277 \\
13	0.00190457392072222 \\
14	0.000304579149567335 \\
15	1.22267286630805e-05 \\
16	2.22663023853651e-08 \\
17	1.51344733324659e-13 \\
};
\addplot [color3, mark=*, mark size=2, mark options={solid}]
table [row sep=\\]{%
0	11.8151246696597 \\
1	10.1421530308999 \\
2	6.62569054780008 \\
3	5.43910480071508 \\
4	5.03755825502935 \\
5	4.18520836791105 \\
6	3.35127523066965 \\
7	1.40531073014416 \\
8	0.501875209730735 \\
9	0.166640362380701 \\
10	0.0657024325461015 \\
11	0.028224311057801 \\
12	0.011899808778143 \\
13	0.00487191036596532 \\
14	0.00189691781695979 \\
15	0.000659182474980718 \\
16	0.000178685863417003 \\
17	2.71780838914374e-05 \\
18	9.9304637175351e-07 \\
19	1.47202653968575e-09 \\
};
\addplot [color4, mark=*, mark size=2, mark options={solid}]
table [row sep=\\]{%
0	11.8165983879483 \\
1	10.1426873258464 \\
2	6.62639001148655 \\
3	5.43974664604429 \\
4	5.03826026146556 \\
5	4.18604409673387 \\
6	3.35271067168767 \\
7	1.40668007074667 \\
8	0.503144176537508 \\
9	0.167789141422784 \\
10	0.0667287464189369 \\
11	0.0291220922689747 \\
12	0.0126658091211486 \\
13	0.00551618167677267 \\
14	0.00242555997290428 \\
15	0.00106339434872046 \\
16	0.000451318967357155 \\
17	0.00017650506636656 \\
18	5.88552321690659e-05 \\
19	1.46012408829587e-05 \\
20	1.90922534964167e-06 \\
21	5.04019885321872e-08 \\
22	3.94970232234965e-11 \\
};
\addplot [color5, mark=*, mark size=2, mark options={solid}]
table [row sep=\\]{%
0	11.8167457609787 \\
1	10.1427407658839 \\
2	6.62645996940205 \\
3	5.43981084403714 \\
4	5.03833047958061 \\
5	4.18612769875065 \\
6	3.35285427729572 \\
7	1.40681713717131 \\
8	0.503271310546687 \\
9	0.167904468672976 \\
10	0.0668323058757353 \\
11	0.0292137518141217 \\
12	0.0127461546103565 \\
13	0.0055879257187261 \\
14	0.00249206585772666 \\
15	0.00112679659478832 \\
16	0.000511425981412238 \\
17	0.000230273981436803 \\
18	0.000101241717481714 \\
19	4.22942785723121e-05 \\
20	1.59688553579993e-05 \\
21	5.01644057754927e-06 \\
22	1.13265893265726e-06 \\
23	1.24084281522852e-07 \\
24	2.18505682493516e-09 \\
};
\addplot [color6, mark=*, mark size=2, mark options={solid}]
table [row sep=\\]{%
0	11.8167604982938 \\
1	10.1427461099933 \\
2	6.62646696530873 \\
3	5.43981726397081 \\
4	5.03833750157017 \\
5	4.18613605924307 \\
6	3.35286863848 \\
7	1.40683084514021 \\
8	0.503284026323158 \\
9	0.167916005898925 \\
10	0.0668426711374848 \\
11	0.0292229367718149 \\
12	0.0127542276195229 \\
13	0.00559517840600272 \\
14	0.00249887724688294 \\
15	0.00113347143224056 \\
16	0.000518134679411495 \\
17	0.000237056009700003 \\
18	0.000108006141047801 \\
19	4.87820102942977e-05 \\
20	2.16942262894475e-05 \\
21	9.37053058753504e-06 \\
22	3.81980297334185e-06 \\
23	1.39055383037196e-06 \\
24	4.13038694409014e-07 \\
25	8.50750577570783e-08 \\
26	7.62409438416849e-09 \\
};
\addplot [white!49.80392156862745!black, mark=*, mark size=2, mark options={solid}]
table [row sep=\\]{%
0	11.8167619720254 \\
1	10.1427466444047 \\
2	6.62646766490069 \\
3	5.43981790596651 \\
4	5.03833820377231 \\
5	4.18613689529573 \\
6	3.3528700746002 \\
7	1.40683221594967 \\
8	0.503285297924141 \\
9	0.167917159666522 \\
10	0.0668437077564224 \\
11	0.0292238554575317 \\
12	0.0127550353054999 \\
13	0.00559590446258309 \\
14	0.00249956002121203 \\
15	0.00113414236811883 \\
16	0.000518813026660576 \\
17	0.000237750674353264 \\
18	0.000108718887437457 \\
19	4.95095438388823e-05 \\
20	2.24279700129813e-05 \\
21	1.00937070471798e-05 \\
22	4.50019632916821e-06 \\
23	1.97409564505535e-06 \\
24	8.39153186750794e-07 \\
25	3.34198900355934e-07 \\
26	1.17413210832019e-07 \\
27	3.28999729951426e-08 \\
28	6.0357030386734e-09 \\
};
\addplot [color7, mark=*, mark size=2, mark options={solid}]
table [row sep=\\]{%
0	11.8167621193986 \\
1	10.1427466978469 \\
2	6.62646773485942 \\
3	5.43981797016356 \\
4	5.03833827398716 \\
5	4.18613697889993 \\
6	3.35287021821577 \\
7	1.40683235303113 \\
8	0.503285425085018 \\
9	0.167917275043554 \\
10	0.0668438114195269 \\
11	0.0292239473280551 \\
12	0.0127551160780978 \\
13	0.00559597707667101 \\
14	0.00249962831339597 \\
15	0.00113420949621918 \\
16	0.000518880934991767 \\
17	0.00023782030789815 \\
18	0.000108790533668395 \\
19	4.9583146484188e-05 \\
20	2.25032224658403e-05 \\
21	1.01701735034221e-05 \\
22	4.57742274754485e-06 \\
23	2.05096861524693e-06 \\
24	9.13557634026743e-07 \\
25	4.02816158658852e-07 \\
26	1.74404819210622e-07 \\
27	7.25197003327126e-08 \\
28	2.74283850080552e-08 \\
29	9.21216798534757e-09 \\
};
\addplot [color8, mark=*, mark size=2, mark options={solid}]
table [row sep=\\]{%
0	11.8167621341359 \\
1	10.1427467031925 \\
2	6.62646774185632 \\
3	5.4398179765848 \\
4	5.03833828101445 \\
5	4.18613698726116 \\
6	3.3528702325783 \\
7	1.40683236673972 \\
8	0.503285437800953 \\
9	0.167917286581398 \\
10	0.0668438217855399 \\
11	0.0292239565154916 \\
12	0.0127551241548895 \\
13	0.00559598433864007 \\
14	0.00249963514361617 \\
15	0.00113421620682827 \\
16	0.000518887734093916 \\
17	0.00023782726087916 \\
18	0.000108797714620225 \\
19	4.95905040980193e-05 \\
20	2.25107791438021e-05 \\
21	1.01779146196403e-05 \\
22	4.58521020257625e-06 \\
23	2.05878646318277e-06 \\
24	9.21519577789729e-07 \\
25	4.1083770122275e-07 \\
26	1.82060557130721e-07 \\
27	7.94447067472851e-08 \\
28	3.47880297276285e-08 \\
29	1.40362018156648e-08 \\
30	4.62545669438943e-09 \\
};
\end{axis}

\end{tikzpicture}
        \end{tabular}
        \caption{\label{fig:test3}{Test case 3.} Influence of the lower bound  $\alpha$ in the sum of square interpolation of $p(x) = (T_{11}(x) + 1) + \alpha$. (Top left) Number of iterations to converge vs. $\alpha$; (Top right) Condition number of $H_m$ at the last iteration vs. $\alpha$; (Bottom) Error $\|\nabla G(\lambda_m)\|_2$ vs. iteration $m$ for different lower bounds $\alpha$.}
    \end{figure} 
    
    Now, we  illustrate the influence of the lower bound of $p$ on the convergence of the method. To proceed, we compute a sum of squares approximation of the polynomial $p(x) = T_{11}(x) + 1 + \alpha$ for various lower bounds $\alpha$
    ($n=5$, $r_*=i_*=n+1=6$, $j_*=2$). 
    \newline
    The results are displayed on Figure~\ref{fig:test3}. We observe that the number of iterations  required to reach a precision of $10^{-8}$ seems to increase proportionally with $ \left| \log(\alpha) \right|$. The condition number of $H_m$ and the norm of $\lambda_m$ at convergence decays like some negative power of $\alpha$. Interestingly enough, one also sees that the quadratic convergence of the (modified) Newton method seems to degenerate to linear convergence when $\alpha$ goes to $0$. All these behaviors can be interpreted thanks to the results of  Lemma~\ref{l:touchzero} and Lemma~\ref{l:cubic}. We know from Lemma~\ref{l:touchzero} that for $\alpha = 0$, $p$ has a root $x_0$ in $[0,1]$, and thus the coercivity of $G$ is lost in some direction of the asymptotic cone of $\mathcal{D}$ (that of the Lagrange vector $L(x_0)$). Thus as $\alpha\to0$, the minimizer $\lambda_\alpha^*$ may go to $+\infty$ in the asymptotic cone which would explain here the explosion of the norm of $\lambda$ and of the condition number of $H_m$ as predicted by Lemma~\ref{l:cubic} and shown on Figure~\ref{fig:test3}. 
    
    \subsubsection{Test case 4}
    
    In this fourth test case we illustrate the influence of the degree $n$ of the objective polynomial $p(x) = x^n + 1$ on the convergence of our method, with $g_1(x)=x$ and $g_2(x)=1-x$ for $n$ odd and  $g_1(x)=1$ and $g_2(x)=x(1-x)$ for $n$ even.
    \newline
    The result are displayed on Figure~\ref{fig:test4}. We observe that the number of iterations required to reach an error of $10^{-8}$  increases 
    with the degree, but weakly. We also 
    observe that the condition number   $\mbox{cond}(H_m)=\|H_m\| \|H_m^{-1}\|$ near convergence deteriorates  with $n$, approximately quadratically.

    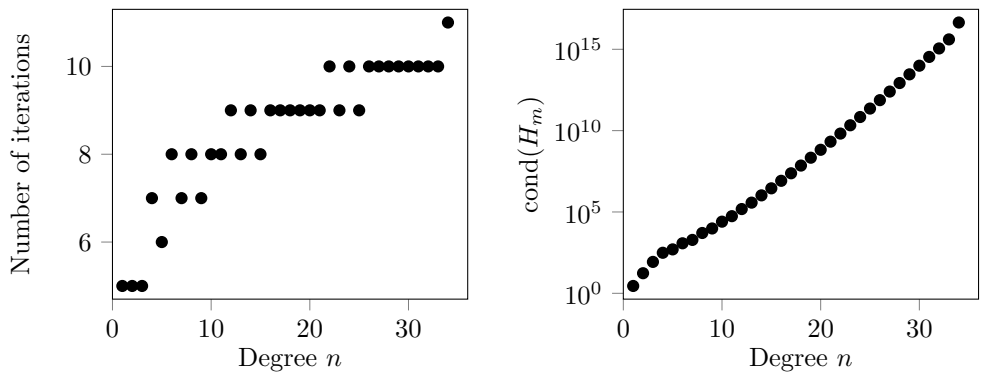
\begin{figure}
        \begin{tabular}{cc}
\begin{tikzpicture}

\begin{axis}[
tick align=outside,
tick pos=left,
x grid style={white!69.01960784313725!black},
xlabel={Degree $n$},
xmin=-0, xmax=36,
y grid style={white!69.01960784313725!black},
ylabel={Number of iterations},
ymin=4.7, ymax=11.3,
width = .4\textwidth
]
\addplot [semithick, black, mark=*, mark size=2, mark options={solid}, forget plot]
table [row sep=\\]{%
1	5 \\
};
\addplot [semithick, black, mark=*, mark size=2, mark options={solid}, forget plot]
table [row sep=\\]{%
2	5 \\
};
\addplot [semithick, black, mark=*, mark size=2, mark options={solid}, forget plot]
table [row sep=\\]{%
3	5 \\
};
\addplot [semithick, black, mark=*, mark size=2, mark options={solid}, forget plot]
table [row sep=\\]{%
4	7 \\
};
\addplot [semithick, black, mark=*, mark size=2, mark options={solid}, forget plot]
table [row sep=\\]{%
5	6 \\
};
\addplot [semithick, black, mark=*, mark size=2, mark options={solid}, forget plot]
table [row sep=\\]{%
6	8 \\
};
\addplot [semithick, black, mark=*, mark size=2, mark options={solid}, forget plot]
table [row sep=\\]{%
7	7 \\
};
\addplot [semithick, black, mark=*, mark size=2, mark options={solid}, forget plot]
table [row sep=\\]{%
8	8 \\
};
\addplot [semithick, black, mark=*, mark size=2, mark options={solid}, forget plot]
table [row sep=\\]{%
9	7 \\
};
\addplot [semithick, black, mark=*, mark size=2, mark options={solid}, forget plot]
table [row sep=\\]{%
10	8 \\
};
\addplot [semithick, black, mark=*, mark size=2, mark options={solid}, forget plot]
table [row sep=\\]{%
11	8 \\
};
\addplot [semithick, black, mark=*, mark size=2, mark options={solid}, forget plot]
table [row sep=\\]{%
12	9 \\
};
\addplot [semithick, black, mark=*, mark size=2, mark options={solid}, forget plot]
table [row sep=\\]{%
13	8 \\
};
\addplot [semithick, black, mark=*, mark size=2, mark options={solid}, forget plot]
table [row sep=\\]{%
14	9 \\
};
\addplot [semithick, black, mark=*, mark size=2, mark options={solid}, forget plot]
table [row sep=\\]{%
15	8 \\
};
\addplot [semithick, black, mark=*, mark size=2, mark options={solid}, forget plot]
table [row sep=\\]{%
16	9 \\
};
\addplot [semithick, black, mark=*, mark size=2, mark options={solid}, forget plot]
table [row sep=\\]{%
17	9 \\
};
\addplot [semithick, black, mark=*, mark size=2, mark options={solid}, forget plot]
table [row sep=\\]{%
18	9 \\
};
\addplot [semithick, black, mark=*, mark size=2, mark options={solid}, forget plot]
table [row sep=\\]{%
19	9 \\
};
\addplot [semithick, black, mark=*, mark size=2, mark options={solid}, forget plot]
table [row sep=\\]{%
20	9 \\
};
\addplot [semithick, black, mark=*, mark size=2, mark options={solid}, forget plot]
table [row sep=\\]{%
21	9 \\
};
\addplot [semithick, black, mark=*, mark size=2, mark options={solid}, forget plot]
table [row sep=\\]{%
22	10 \\
};
\addplot [semithick, black, mark=*, mark size=2, mark options={solid}, forget plot]
table [row sep=\\]{%
23	9 \\
};
\addplot [semithick, black, mark=*, mark size=2, mark options={solid}, forget plot]
table [row sep=\\]{%
24	10 \\
};
\addplot [semithick, black, mark=*, mark size=2, mark options={solid}, forget plot]
table [row sep=\\]{%
25	9 \\
};
\addplot [semithick, black, mark=*, mark size=2, mark options={solid}, forget plot]
table [row sep=\\]{%
26	10 \\
};
\addplot [semithick, black, mark=*, mark size=2, mark options={solid}, forget plot]
table [row sep=\\]{%
27	10 \\
};
\addplot [semithick, black, mark=*, mark size=2, mark options={solid}, forget plot]
table [row sep=\\]{%
28	10 \\
};
\addplot [semithick, black, mark=*, mark size=2, mark options={solid}, forget plot]
table [row sep=\\]{%
29	10 \\
};
\addplot [semithick, black, mark=*, mark size=2, mark options={solid}, forget plot]
table [row sep=\\]{%
30	10 \\
};
\addplot [semithick, black, mark=*, mark size=2, mark options={solid}, forget plot]
table [row sep=\\]{%
31	10 \\
};
\addplot [semithick, black, mark=*, mark size=2, mark options={solid}, forget plot]
table [row sep=\\]{%
32	10 \\
};
\addplot [semithick, black, mark=*, mark size=2, mark options={solid}, forget plot]
table [row sep=\\]{%
33	10 \\
};
\addplot [semithick, black, mark=*, mark size=2, mark options={solid}, forget plot]
table [row sep=\\]{%
34	11 \\
};
\end{axis}

\end{tikzpicture}
            &
\begin{tikzpicture}

\begin{axis}[
tick align=outside,
tick pos=left,
x grid style={white!69.01960784313725!black},
xlabel={Degree $n$},
xmin=0, xmax=36,
y grid style={white!69.01960784313725!black},
ylabel={$\mathrm{cond}(H_m)$},
ymin=0.438305373738238, ymax=2.86208209753563e+17,
ymode=log,
width = .4\textwidth
]
\addplot [semithick, black, mark=*, mark size=2, mark options={solid}, forget plot]
table [row sep=\\]{%
1	2.82842858432658 \\
};
\addplot [semithick, black, mark=*, mark size=2, mark options={solid}, forget plot]
table [row sep=\\]{%
2	17.164143879666 \\
};
\addplot [semithick, black, mark=*, mark size=2, mark options={solid}, forget plot]
table [row sep=\\]{%
3	85.2755656217145 \\
};
\addplot [semithick, black, mark=*, mark size=2, mark options={solid}, forget plot]
table [row sep=\\]{%
4	310.437958648156 \\
};
\addplot [semithick, black, mark=*, mark size=2, mark options={solid}, forget plot]
table [row sep=\\]{%
5	500.767045328044 \\
};
\addplot [semithick, black, mark=*, mark size=2, mark options={solid}, forget plot]
table [row sep=\\]{%
6	1187.13996783804 \\
};
\addplot [semithick, black, mark=*, mark size=2, mark options={solid}, forget plot]
table [row sep=\\]{%
7	1915.54136727739 \\
};
\addplot [semithick, black, mark=*, mark size=2, mark options={solid}, forget plot]
table [row sep=\\]{%
8	5100.51907451592 \\
};
\addplot [semithick, black, mark=*, mark size=2, mark options={solid}, forget plot]
table [row sep=\\]{%
9	9681.74449855338 \\
};
\addplot [semithick, black, mark=*, mark size=2, mark options={solid}, forget plot]
table [row sep=\\]{%
10	25322.9917834682 \\
};
\addplot [semithick, black, mark=*, mark size=2, mark options={solid}, forget plot]
table [row sep=\\]{%
11	55117.6217818487 \\
};
\addplot [semithick, black, mark=*, mark size=2, mark options={solid}, forget plot]
table [row sep=\\]{%
12	150884.00301723 \\
};
\addplot [semithick, black, mark=*, mark size=2, mark options={solid}, forget plot]
table [row sep=\\]{%
13	374762.068820122 \\
};
\addplot [semithick, black, mark=*, mark size=2, mark options={solid}, forget plot]
table [row sep=\\]{%
14	1048975.54093459 \\
};
\addplot [semithick, black, mark=*, mark size=2, mark options={solid}, forget plot]
table [row sep=\\]{%
15	2839677.70773532 \\
};
\addplot [semithick, black, mark=*, mark size=2, mark options={solid}, forget plot]
table [row sep=\\]{%
16	8234125.85762619 \\
};
\addplot [semithick, black, mark=*, mark size=2, mark options={solid}, forget plot]
table [row sep=\\]{%
17	23883411.8028712 \\
};
\addplot [semithick, black, mark=*, mark size=2, mark options={solid}, forget plot]
table [row sep=\\]{%
18	71207307.9412311 \\
};
\addplot [semithick, black, mark=*, mark size=2, mark options={solid}, forget plot]
table [row sep=\\]{%
19	216447049.745579 \\
};
\addplot [semithick, black, mark=*, mark size=2, mark options={solid}, forget plot]
table [row sep=\\]{%
20	664521815.247751 \\
};
\addplot [semithick, black, mark=*, mark size=2, mark options={solid}, forget plot]
table [row sep=\\]{%
21	2094890107.89472 \\
};
\addplot [semithick, black, mark=*, mark size=2, mark options={solid}, forget plot]
table [row sep=\\]{%
22	6590717329.47658 \\
};
\addplot [semithick, black, mark=*, mark size=2, mark options={solid}, forget plot]
table [row sep=\\]{%
23	21348714150.2699 \\
};
\addplot [semithick, black, mark=*, mark size=2, mark options={solid}, forget plot]
table [row sep=\\]{%
24	68770187983.2239 \\
};
\addplot [semithick, black, mark=*, mark size=2, mark options={solid}, forget plot]
table [row sep=\\]{%
25	227577234835.458 \\
};
\addplot [semithick, black, mark=*, mark size=2, mark options={solid}, forget plot]
table [row sep=\\]{%
26	747489175889.979 \\
};
\addplot [semithick, black, mark=*, mark size=2, mark options={solid}, forget plot]
table [row sep=\\]{%
27	2518515305288.15 \\
};
\addplot [semithick, black, mark=*, mark size=2, mark options={solid}, forget plot]
table [row sep=\\]{%
28	8426054738745.84 \\
};
\addplot [semithick, black, mark=*, mark size=2, mark options={solid}, forget plot]
table [row sep=\\]{%
29	28801222548393.4 \\
};
\addplot [semithick, black, mark=*, mark size=2, mark options={solid}, forget plot]
table [row sep=\\]{%
30	97977245277048.7 \\
};
\addplot [semithick, black, mark=*, mark size=2, mark options={solid}, forget plot]
table [row sep=\\]{%
31	339838609491018 \\
};
\addplot [semithick, black, mark=*, mark size=2, mark options={solid}, forget plot]
table [row sep=\\]{%
32	1.12754918674727e+15 \\
};
\addplot [semithick, black, mark=*, mark size=2, mark options={solid}, forget plot]
table [row sep=\\]{%
33	4.05390118629145e+15 \\
};
\addplot [semithick, black, mark=*, mark size=2, mark options={solid}, forget plot]
table [row sep=\\]{%
34	4.43520465880365e+16 \\
};
\end{axis}

\end{tikzpicture}
        \end{tabular}
        \caption{\label{fig:test4}\textbf{Test case 4.} Influence of the degree  $n$ in the sum of square interpolation of $p(x) = x^n + 1$. (Left) Number of iterations to converge vs. $n$; (Right) Condition number of $H_m$ at the last iteration vs. $n$.}
    \end{figure}

    \subsection{Bivariate polynomials on a triangle}
    
    We use the   minimization algorithm   for
    the computation of a sum of squares representation of some positive polynomial $p\in P_n[X,Y]$ on the triangle.
    
    \subsubsection{Numerical setting}   
    
    The barycentric coordinates corresponding to the vertices $S_1$, $S_2$ and $S_3$ of the triangle 
    are denoted as  $\mu_j$ for  $j = 1,2,3$: 
    $
    \mu_1(x,y) = 1 - x - y$, $ \mu_2(x,y) = x$ and $\mu_3(x,y) = y$. 
    The triangle is 
    $
    \mathbb{K} = \{\mathbf x=(x,y)\in \mathbb R^2\mid\ \mu_1(\mathbf x)\geq 0\,,\ \mu_2(\mathbf x)\geq 0\,,\ \mu_3(\mathbf x)\geq 0\}$. 
    The interpolation points are
    $\mathbf x_r=(x_r, y_r)$  for  $1\leq r \leq r_* = (n+1)(n+2)/2$ are the distinct points of a cartesian grid intersected with
    the triangle. 
    For a given polynomial $p\in \sP^n[\bX]$ of a given degree, 
    the data is    $\mathbf z\in\mathbb{R}^{r_*}$ which is
    the vector with components $z_r = p(x_r, y_r)$.
    An illustration of the geometry is provided in  Figure~\ref{fig:simplex} where 
    the degree is $n=4$. 
    
    \begin{figure}
        \centering
        \begin{tabular}{cc}
            \scalebox{.9}{   
\begin{tikzpicture}

\definecolor{color0}{rgb}{0.12156862745098,0.466666666666667,0.705882352941177}

\begin{axis}[
tick align=outside,
tick pos=left,
title={Domain and interpolation points},
x grid style={white!69.01960784313725!black},
xlabel={$x$},
axis equal,
xmin=-0.2, xmax=1.2,
y grid style={white!69.01960784313725!black},
ylabel={$y$},
ymin=-0.2, ymax=1.2,
width = .5\textwidth
]
\addplot [very thick, black, forget plot]
table [row sep=\\]{%
0	0 \\
1	0 \\
};
\addplot [very thick, black, forget plot]
table [row sep=\\]{%
0	0 \\
0	1 \\
};
\addplot [very thick, black, forget plot]
table [row sep=\\]{%
0	1 \\
1	0 \\
};
\addplot [semithick, color0, mark=*, mark size=3, mark options={solid,fill=black,draw=white}, only marks, forget plot]
table [row sep=\\]{%
0	0 \\
0.25	0 \\
0.5	0 \\
0.75	0 \\
1	0 \\
0	0.25 \\
0.25	0.25 \\
0.5	0.25 \\
0.75	0.25 \\
0	0.5 \\
0.25	0.5 \\
0.5	0.5 \\
0	0.75 \\
0.25	0.75 \\
0	1 \\
};
\node at (axis cs:-0.12,-0.12)[
  scale=1,
  anchor=base west,
  text=black,
  rotate=0.0
]{ $S_1$};
\node at (axis cs:1,-0.12)[
  scale=1,
  anchor=base west,
  text=black,
  rotate=0.0
]{ $S_2$};
\node at (axis cs:-0.1,1.1)[
  scale=1,
  anchor=base west,
  text=black,
  rotate=0.0
]{ $S_3$};
\end{axis}

\end{tikzpicture} }
        \end{tabular}
        \caption{\label{fig:simplex}The simplex $\mathbb{K}$ and interpolation points for $n = 4$.}
    \end{figure}
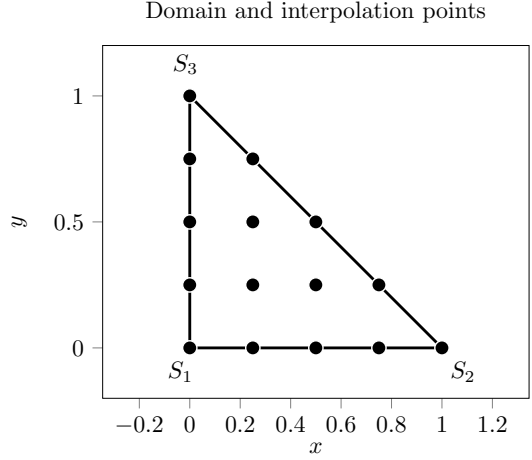
    
    We consider  the   ansatz ($\mathbf x=(x,y)$)
    \begin{equation} \label{e:sos2d}
        p[\lambda](\mathbf  x)\ =\ \sum_{i = 1}^{r_j} g_1(\mathbf x)\,p_{i1}[\lambda](\mathbf x)^2 +  g_2(\mathbf x)\,p_{i2}[\lambda](\mathbf  x)^2 
        + g_3(\mathbf  x)\,p_{i3}[\lambda](\mathbf  x)^2 + g_4(\mathbf  x)\,p_{i4}[\lambda](\mathbf x)^2, 
    \end{equation}
    where,  arbitrarily with respect to the literature \cite{lasserre_2010_moments},  the  weights are
    \begin{equation}\label{e:sos2drab}
        \left\{
        \begin{array}{ll}
            \mbox{for } n=2k+1, &
            g_i = \mu_i $ for $ i = 1,2,3 $ and $ g_4 = \mu_1\mu_2\mu_3, \\
            \mbox{for } n=2k, &
            g_1 = \mu_2\,\mu_3, \ g_2 = \mu_3\,\mu_1, \ g_3 = \mu_1\,\mu_2 $ and $ g_4 = 1.
        \end{array}
        \right.
    \end{equation}
    With this choice we recover in every cases $r_* = r_1 + r_2 + r_3 + r_4$. All polynomials are parametrized on the basis of bivariate monomials
    since  Chebychev polynomials are not available on the triangle.

    \subsubsection{Test case 5}   
    We  approach the polynomial $p(x,y) = (T_{4}(x) + 1)(T_{4}(y) + 1) / 4 + 10^{-3}$ on the 2D simplex with the modified Newton method. 
    The  parameters are $n=8$, $r_*=i_*=45$ and $j_*=4$.

    We observe on Figure~\ref{fig:test5} that our method converges in this multivariate setting and reaches a precision of less than $10^{-8}$ in $210$ iterations. 
    The error decays slowly during the first $200$ iterations before reaching usual quadratic speed of convergence of the Newton method near the minimizer of $G$.
    This result illustrates the ability of our algorithms to provided a computational strategy for  the computation of positive
    polynomials on bi-dimensional sets.
    
    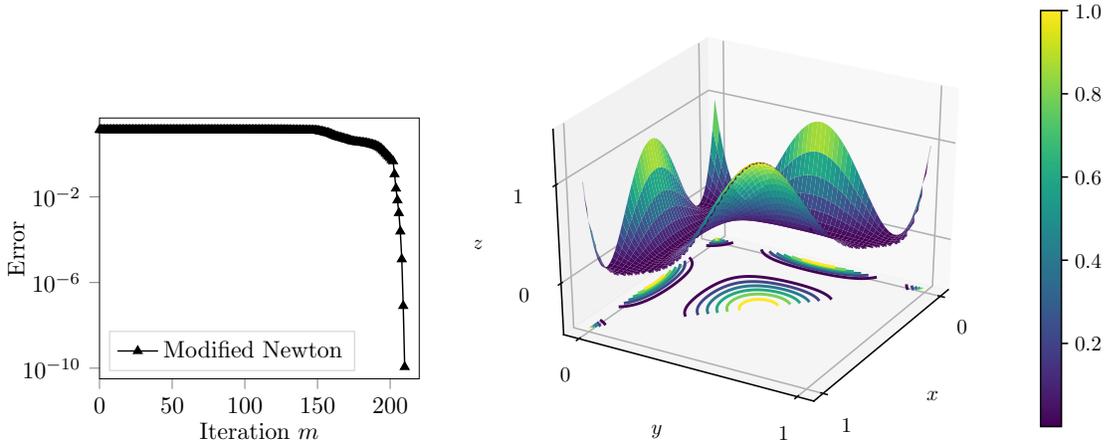
\begin{figure}
        \begin{tabular}{cc}
            \scalebox{.9}{  
\begin{tikzpicture}

\definecolor{color0}{rgb}{0.12156862745098,0.466666666666667,0.705882352941177}

\begin{axis}[
legend cell align={left},
legend entries={{Modified Newton}},
legend style={at={(0.03,0.03)}, anchor=south west, draw=white!80.0!black},
tick align=outside,
tick pos=left,
x grid style={white!69.01960784313725!black},
xlabel={Iteration $m$},
xmin=0, xmax=220,
y grid style={white!69.01960784313725!black},
ylabel={Error},
ymin=3.17781480265495e-11, ymax=48.4849906694354,
ymode=log,
width = .4\textwidth
]
\addplot [semithick, black, mark=triangle*, mark size=2, mark options={solid}]
table [row sep=\\]{%
0	13.5460548102014 \\
1	13.5457930718776 \\
2	13.5455353691784 \\
3	13.5452812801314 \\
4	13.5450307912337 \\
5	13.5447839733317 \\
6	13.5445405577759 \\
7	13.5443004801361 \\
8	13.5440637217092 \\
9	13.5438301954861 \\
10	13.5435996778747 \\
11	13.5433722600293 \\
12	13.54314782788 \\
13	13.5429263331091 \\
14	13.5427077237228 \\
15	13.5424918738972 \\
16	13.5422787032958 \\
17	13.5420682103267 \\
18	13.5418601806759 \\
19	13.5416545942641 \\
20	13.5414514869102 \\
21	13.5412506716166 \\
22	13.5410519816955 \\
23	13.5408554648065 \\
24	13.5406608344564 \\
25	13.5404679084428 \\
26	13.5402766309722 \\
27	13.5400870137267 \\
28	13.5398986556449 \\
29	13.5397116610878 \\
30	13.5395256811867 \\
31	13.5393406380395 \\
32	13.5391564608561 \\
33	13.5389727019106 \\
34	13.5387892029719 \\
35	13.538605881336 \\
36	13.53842238506 \\
37	13.5382386844019 \\
38	13.5380545103271 \\
39	13.5378695055151 \\
40	13.537683519742 \\
41	13.5374961625462 \\
42	13.5373073243293 \\
43	13.5371165253682 \\
44	13.5369236808368 \\
45	13.5367284002749 \\
46	13.5365304166027 \\
47	13.5363293541295 \\
48	13.5361249075929 \\
49	13.5359165592581 \\
50	13.5357042244236 \\
51	13.5354872698082 \\
52	13.5352654083637 \\
53	13.5350386383327 \\
54	13.5348058999624 \\
55	13.5345671395365 \\
56	13.534321711437 \\
57	13.5340692626868 \\
58	13.5338091401595 \\
59	13.5335410109077 \\
60	13.5332643224897 \\
61	13.5329784762552 \\
62	13.5326827613458 \\
63	13.5323767588988 \\
64	13.5320598414097 \\
65	13.5317310447798 \\
66	13.5313898942012 \\
67	13.531035686717 \\
68	13.5306673595858 \\
69	13.5302841773774 \\
70	13.5298853245787 \\
71	13.5294697891563 \\
72	13.5290366971101 \\
73	13.5285849107738 \\
74	13.5281133752554 \\
75	13.5276207739112 \\
76	13.5271060894196 \\
77	13.526567826744 \\
78	13.5260045307531 \\
79	13.5254146656175 \\
80	13.5247965182332 \\
81	13.5241485326026 \\
82	13.5234686704971 \\
83	13.5227549557898 \\
84	13.5220051912605 \\
85	13.5212170680404 \\
86	13.5203879962086 \\
87	13.519515272626 \\
88	13.518595943105 \\
89	13.5176267935031 \\
90	13.5166043351713 \\
91	13.5155248390745 \\
92	13.51438409168 \\
93	13.513177687458 \\
94	13.5119008167828 \\
95	13.510548028847 \\
96	13.5091135709052 \\
97	13.5075909155723 \\
98	13.5059730705669 \\
99	13.5042522203558 \\
100	13.5024198406326 \\
101	13.5004664402306 \\
102	13.498381460883 \\
103	13.4961532001729 \\
104	13.4937687014863 \\
105	13.4912134550201 \\
106	13.4884711990576 \\
107	13.4855237234195 \\
108	13.4823505835768 \\
109	13.4789287131185 \\
110	13.4752320111401 \\
111	13.4712308905124 \\
112	13.4668917887235 \\
113	13.4621763667763 \\
114	13.4570408485039 \\
115	13.4514350485609 \\
116	13.4453012063789 \\
117	13.4385727747359 \\
118	13.4311727783226 \\
119	13.4230120442585 \\
120	13.4139870213473 \\
121	13.4039772994515 \\
122	13.3928428099107 \\
123	13.3804207908027 \\
124	13.3665227615952 \\
125	13.3509321025349 \\
126	13.3334034638963 \\
127	13.3136664170873 \\
128	13.291438772235 \\
129	13.2664598692852 \\
130	13.2385663469337 \\
131	13.2078573433211 \\
132	13.17505051077 \\
133	13.1422486819044 \\
134	13.114593595675 \\
135	13.1037572263701 \\
136	13.0706180989241 \\
137	13.0422842789077 \\
138	13.0225248539671 \\
139	13.0153134308689 \\
140	12.9812332589467 \\
141	12.9465100422643 \\
142	12.9102445340284 \\
143	12.8710800738856 \\
144	12.8273038760399 \\
145	12.8206670674392 \\
146	12.7761836886664 \\
147	12.6803643900895 \\
148	12.6691321601899 \\
149	12.4986781997385 \\
150	12.4449008107458 \\
151	11.9648510140037 \\
152	11.5860599403486 \\
153	11.2849101141779 \\
154	10.9350733279493 \\
155	10.437764013751 \\
156	9.85299499384731 \\
157	9.83944277207862 \\
158	9.32164153982014 \\
159	8.73423020496503 \\
160	8.15736673112223 \\
161	7.57528773941168 \\
162	7.08379726308435 \\
163	6.68982157862273 \\
164	6.35914546534532 \\
165	6.07517499495364 \\
166	5.8352561183057 \\
167	5.61252951048726 \\
168	5.35000929031135 \\
169	5.06665367272299 \\
170	4.84243364773581 \\
171	4.72703308781236 \\
172	4.4248777058418 \\
173	4.18948756816768 \\
174	4.01079796063467 \\
175	3.87516671261542 \\
176	3.76945947870652 \\
177	3.68362096421485 \\
178	3.61043486187959 \\
179	3.54449512174078 \\
180	3.48133672628686 \\
181	3.41683303366355 \\
182	3.34687820639453 \\
183	3.26752796037637 \\
184	3.17595470010773 \\
185	3.07220142567051 \\
186	2.96007193615972 \\
187	2.84469497225161 \\
188	2.72806784215555 \\
189	2.60713408485163 \\
190	2.47538495336446 \\
191	2.32555396105714 \\
192	2.15214891318623 \\
193	1.95381267018989 \\
194	1.73485032704134 \\
195	1.50530393976738 \\
196	1.2795588344956 \\
197	1.07375610064114 \\
198	0.902759744144218 \\
199	0.772193028560252 \\
200	0.659142616855908 \\
201	0.531734591236007 \\
202	0.468144029352018 \\
203	0.116017681049644 \\
204	0.0248014643174052 \\
205	0.00698292366903169 \\
206	0.00177225474816946 \\
207	0.000245136123113443 \\
208	1.22973226150009e-05 \\
209	8.12619575830373e-08 \\
210	1.13742579086485e-10 \\
};
\end{axis}

\end{tikzpicture} }
            &
            \scalebox{.9}{   \includegraphics[width = .6\linewidth]{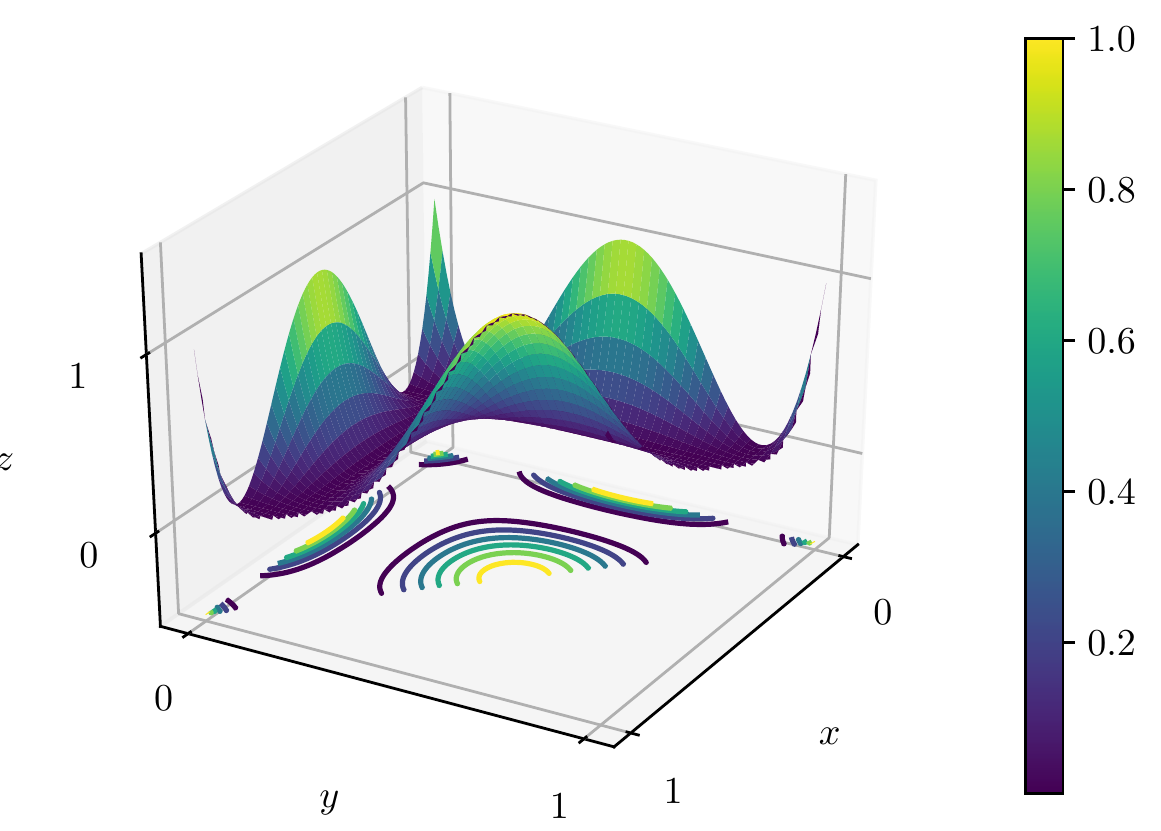} }
        \end{tabular}
        \caption{\label{fig:test5}{Test case 5.} Bivariate sum of square interpolation of the degree $8$ polynomial $p(x,y) = (T_{4}(x) + 1)(T_{4}(y) + 1) / 4 + 10^{-3}$ on the 2D simplex. (Left) error $\|\nabla G(\lambda_m)\|_2$ vs. iteration $m$; (Right) surface plot of the converged sum of square.}
    \end{figure}
    
    \subsubsection{Test case 6} 
    In this last test case we are interested in the SOS approximation of the Motzkin polynomial \cite{motzkin_1967_arithmetic}
    $p(x,y) = x^2y^4 + y^2x^4 - 3x^2y^2 + 1$. 
    
    This polynomial is non-negative over $\mathbb{R}^2$ and  famous for not being a sum of square in the sense that it admits no decomposition \eqref{e:sos} with weigths $\widetilde g_1 = \dots = \widetilde g_{j_*} = 1$ (whatever the choice of $i_*$ or, equivalently in this particular case, $j_*$).
    The parameters are $n=6$, $r_*=i_*=28$ and $j_*=4$.
    We use our method to approach this polynomial with the sum of square ansatz \eqref{e:sos2d} but with
    two different weights: on the one hand we use the weights
    $g_i$ \eqref{e:sos2drab} for which we expect some convergence of the
    algorithm;  on the other other hand we use   the weights  $\widetilde g_i = 1$ for  $i=1,2,3,4$. 
    
    In the latter case  our experiment   on Figure~\ref{fig:test6} show
     the method does not converge 
     (in coherence with  the non-existence of a sum of square decomposition for the Motzkin polynomial). The algorithm with weights $g_i$ converges while the algorithm with weights $\widetilde g_i$ does not converge (bottom right illustration in the Figure).
    
    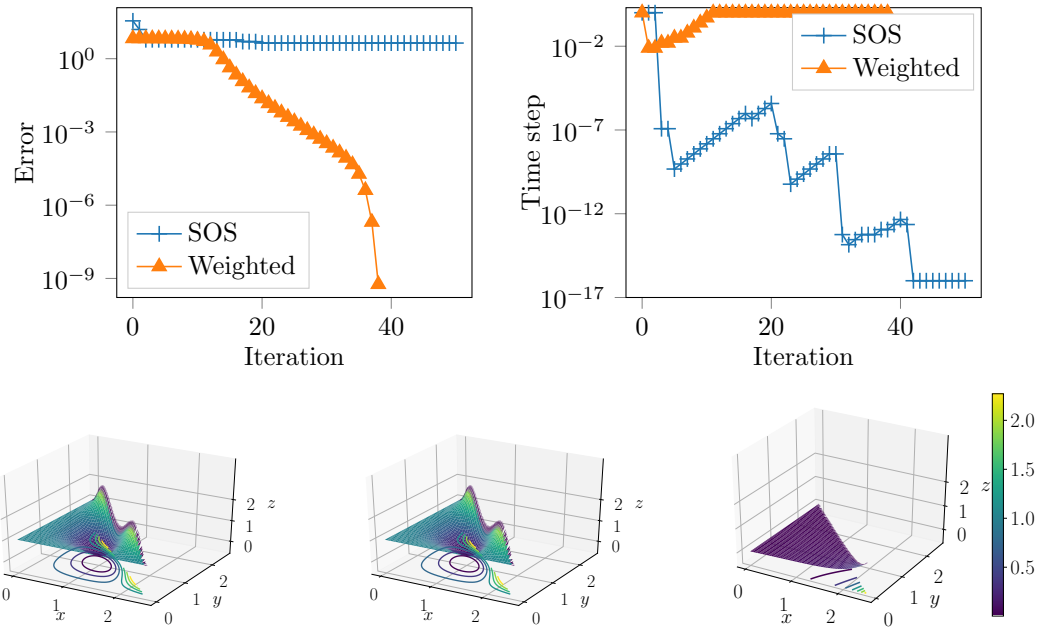
\begin{figure}
        \begin{tabular}{c}
            \begin{tabular}{cc}
\begin{tikzpicture}

\definecolor{color0}{rgb}{0.12156862745098,0.466666666666667,0.705882352941177}
\definecolor{color1}{rgb}{1,0.498039215686275,0.0549019607843137}

\begin{axis}[
legend cell align={left},
legend entries={{SOS},{Weighted}},
legend style={at={(0.03,0.03)}, anchor=south west, draw=white!80.0!black},
tick align=outside,
tick pos=left,
x grid style={white!69.01960784313725!black},
xlabel={Iteration},
xmin=-2.5, xmax=52.5,
y grid style={white!69.01960784313725!black},
ylabel={Error},
ymin=1.65664285144326e-10, ymax=122.918716879743,
ymode=log,
width = .4\textwidth
]
\addplot [semithick, color0, mark=+, mark size=3, mark options={solid}]
table [row sep=\\]{%
0	35.4858225066571 \\
1	15.7026054008647 \\
2	6.00116012515238 \\
3	5.99821921921229 \\
4	5.99787539350365 \\
5	5.97325608039234 \\
6	5.97317250926151 \\
7	5.9731679206387 \\
8	5.97304871594639 \\
9	5.97302490429651 \\
10	5.97296726347928 \\
11	5.97116809975121 \\
12	5.96409323477348 \\
13	5.95810006830177 \\
14	5.94107123352348 \\
15	5.92845126895247 \\
16	5.90811211127213 \\
17	5.03899109451993 \\
18	4.99336546451801 \\
19	4.99119998073611 \\
20	4.37318991146835 \\
21	4.37311657995653 \\
22	4.37304263296317 \\
23	4.37303167746167 \\
24	4.37301526315236 \\
25	4.37285869427849 \\
26	4.37267271843344 \\
27	4.37266406889025 \\
28	4.37261204606337 \\
29	4.37242880651788 \\
30	4.37226713393097 \\
31	4.3721813901094 \\
32	4.37205888736938 \\
33	4.37204441669776 \\
34	4.3719538996089 \\
35	4.37190940854042 \\
36	4.37188480101781 \\
37	4.37177607664018 \\
38	4.37171503015207 \\
39	4.37166188614172 \\
40	4.37152204464229 \\
41	4.37147002498461 \\
42	4.37187219235468 \\
43	4.37166450747394 \\
44	4.37175500802825 \\
45	4.37151845868666 \\
46	4.37174000687502 \\
47	4.37159228782315 \\
48	4.37181437389576 \\
49	4.3717352905742 \\
50	4.37197722987192 \\
};
\addplot [semithick, color1, mark=triangle*, mark size=3, mark options={solid}]
table [row sep=\\]{%
0	6.74072073220441 \\
1	6.72774601916812 \\
2	6.71367029077907 \\
3	6.70153549599155 \\
4	6.67353327490456 \\
5	6.65776358064968 \\
6	6.60322941963257 \\
7	6.53659938416685 \\
8	6.43239001728382 \\
9	6.22592417174114 \\
10	5.85638708022942 \\
11	5.16262356067708 \\
12	3.71595436254755 \\
13	1.92387654413346 \\
14	0.925230616560257 \\
15	0.43064538122105 \\
16	0.219263446476221 \\
17	0.117432931558611 \\
18	0.0649078664716494 \\
19	0.0379049344697125 \\
20	0.0232215582438197 \\
21	0.0146719662713159 \\
22	0.00944294608167531 \\
23	0.00614734342954125 \\
24	0.0040318257805031 \\
25	0.00265759937679241 \\
26	0.00175762191494543 \\
27	0.00116478592030761 \\
28	0.000772522323339002 \\
29	0.000511930461075967 \\
30	0.000337990457385731 \\
31	0.00022102438104997 \\
32	0.000141327531655703 \\
33	8.5814817809831e-05 \\
34	4.62939126169332e-05 \\
35	1.88909030255436e-05 \\
36	4.10029213106541e-06 \\
37	2.03706232393829e-07 \\
38	5.73841605585451e-10 \\
};
\end{axis}

\end{tikzpicture}
                &
\begin{tikzpicture}

\definecolor{color0}{rgb}{0.12156862745098,0.466666666666667,0.705882352941177}
\definecolor{color1}{rgb}{1,0.498039215686275,0.0549019607843137}

\begin{axis}[
legend cell align={left},
legend entries={{SOS},{Weighted}},
legend style={draw=white!80.0!black},
tick align=outside,
tick pos=left,
x grid style={white!69.01960784313725!black},
xlabel={Iteration},
xmin=-2.5, xmax=52.5,
y grid style={white!69.01960784313725!black},
ylabel={Time step},
ymin=1e-17, ymax=2,
ymode=log,
width = .4\textwidth
]
\addplot [semithick, color0, mark=+, mark size=3, mark options={solid}]
table [row sep=\\]{%
0	1 \\
1	1 \\
2	1 \\
3	1.19209289550781e-07 \\
4	1.19209289550781e-07 \\
5	4.65661287307739e-10 \\
6	9.31322574615479e-10 \\
7	1.86264514923096e-09 \\
8	3.72529029846191e-09 \\
9	7.45058059692383e-09 \\
10	1.49011611938477e-08 \\
11	2.98023223876953e-08 \\
12	5.96046447753906e-08 \\
13	1.19209289550781e-07 \\
14	2.38418579101562e-07 \\
15	4.76837158203125e-07 \\
16	9.5367431640625e-07 \\
17	4.76837158203125e-07 \\
18	9.5367431640625e-07 \\
19	1.9073486328125e-06 \\
20	3.814697265625e-06 \\
21	5.96046447753906e-08 \\
22	2.98023223876953e-08 \\
23	5.82076609134674e-11 \\
24	1.16415321826935e-10 \\
25	2.3283064365387e-10 \\
26	4.65661287307739e-10 \\
27	9.31322574615479e-10 \\
28	1.86264514923096e-09 \\
29	3.72529029846191e-09 \\
30	3.72529029846191e-09 \\
31	5.6843418860808e-14 \\
32	1.4210854715202e-14 \\
33	2.8421709430404e-14 \\
34	5.6843418860808e-14 \\
35	5.6843418860808e-14 \\
36	5.6843418860808e-14 \\
37	1.13686837721616e-13 \\
38	1.13686837721616e-13 \\
39	2.27373675443232e-13 \\
40	4.54747350886464e-13 \\
41	2.27373675443232e-13 \\
42	1e-16 \\
43	1e-16 \\
44	1e-16 \\
45	1e-16 \\
46	1e-16 \\
47	1e-16 \\
48	1e-16 \\
49	1e-16 \\
50	1e-16 \\
};
\addplot [semithick, color1, mark=triangle*, mark size=3, mark options={solid}]
table [row sep=\\]{%
0	1 \\
1	0.0078125 \\
2	0.0078125 \\
3	0.015625 \\
4	0.015625 \\
5	0.03125 \\
6	0.03125 \\
7	0.0625 \\
8	0.125 \\
9	0.25 \\
10	0.5 \\
11	1 \\
12	1 \\
13	1 \\
14	1 \\
15	1 \\
16	1 \\
17	1 \\
18	1 \\
19	1 \\
20	1 \\
21	1 \\
22	1 \\
23	1 \\
24	1 \\
25	1 \\
26	1 \\
27	1 \\
28	1 \\
29	1 \\
30	1 \\
31	1 \\
32	1 \\
33	1 \\
34	1 \\
35	1 \\
36	1 \\
37	1 \\
38	1 \\
};
\end{axis}

\end{tikzpicture}
            \end{tabular}\\[1em]
            \begin{tabular}{ccc}
                \scalebox{.9}{    \includegraphics[width = .3\textwidth]{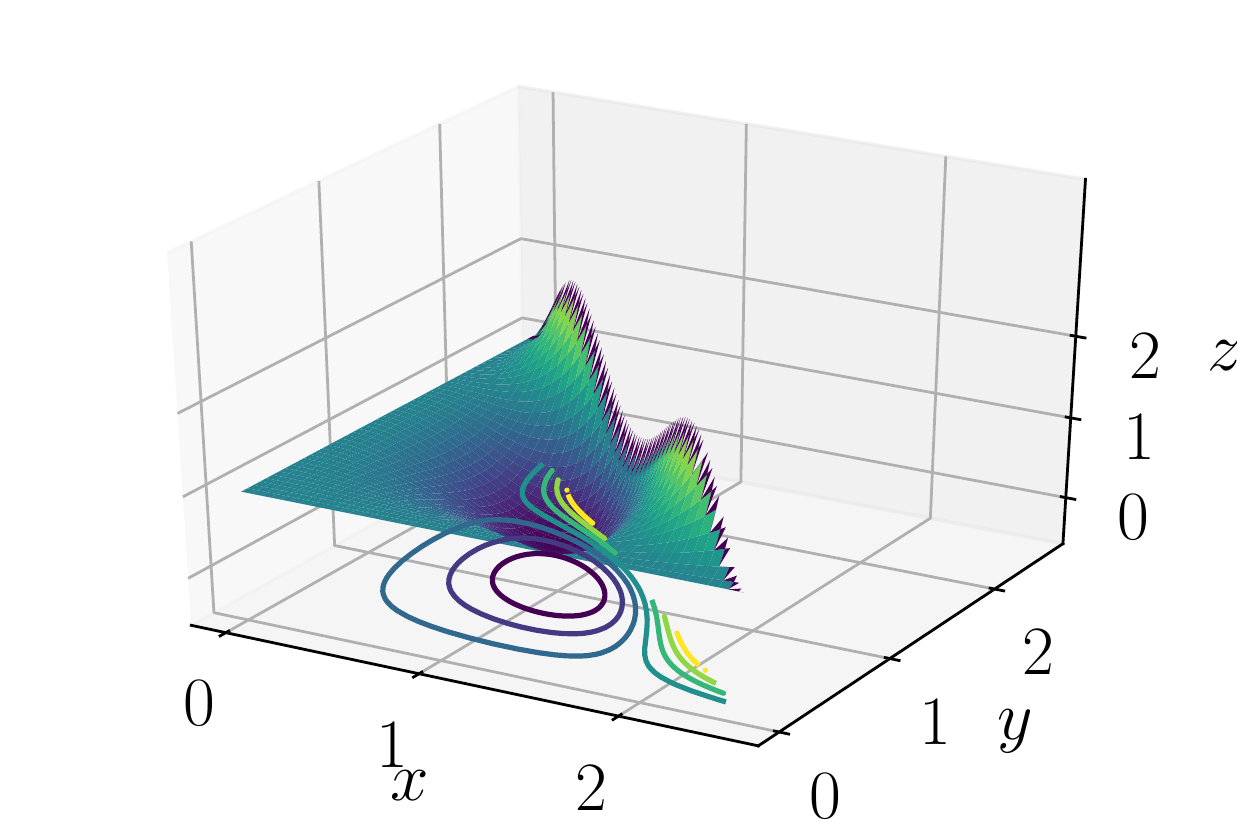} }
                &
                \scalebox{.9}{    \includegraphics[width = .3\textwidth]{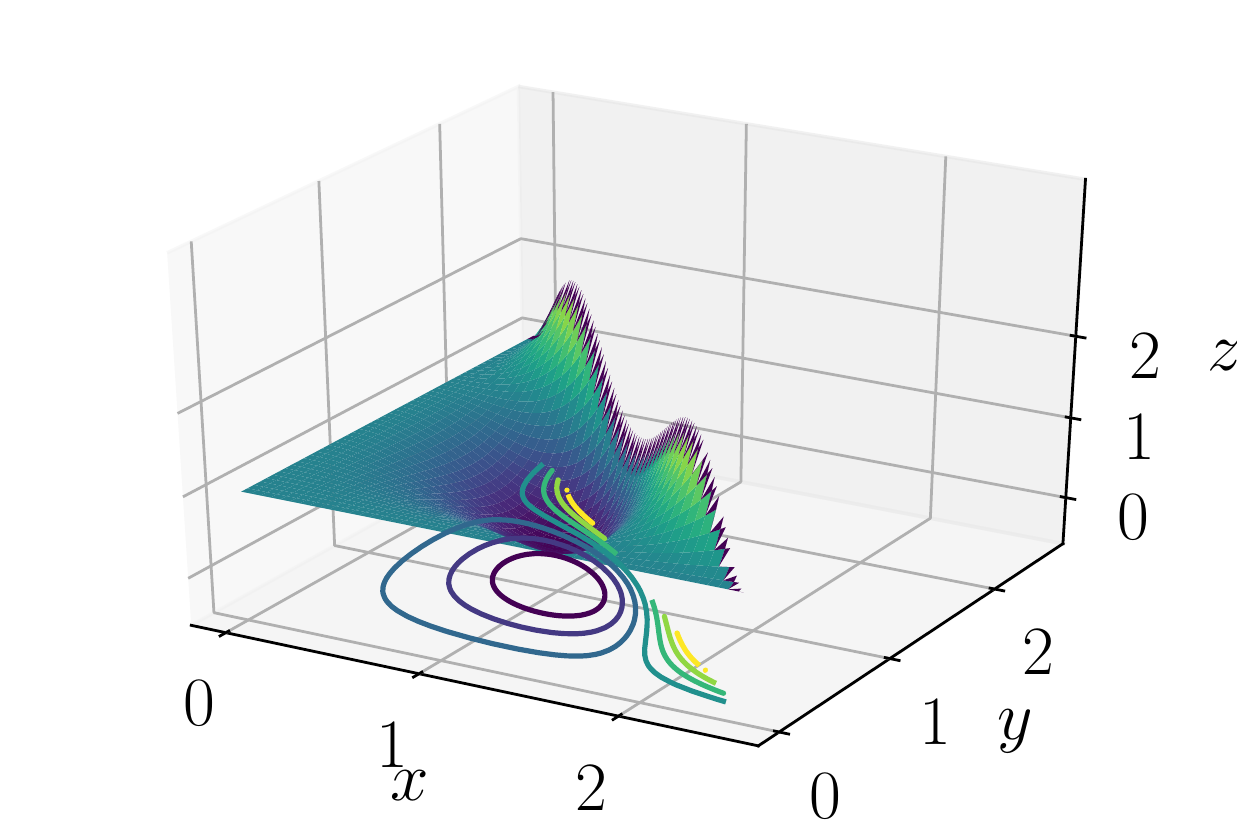} }
                &
                \scalebox{.9}{      \includegraphics[width = .33\textwidth]{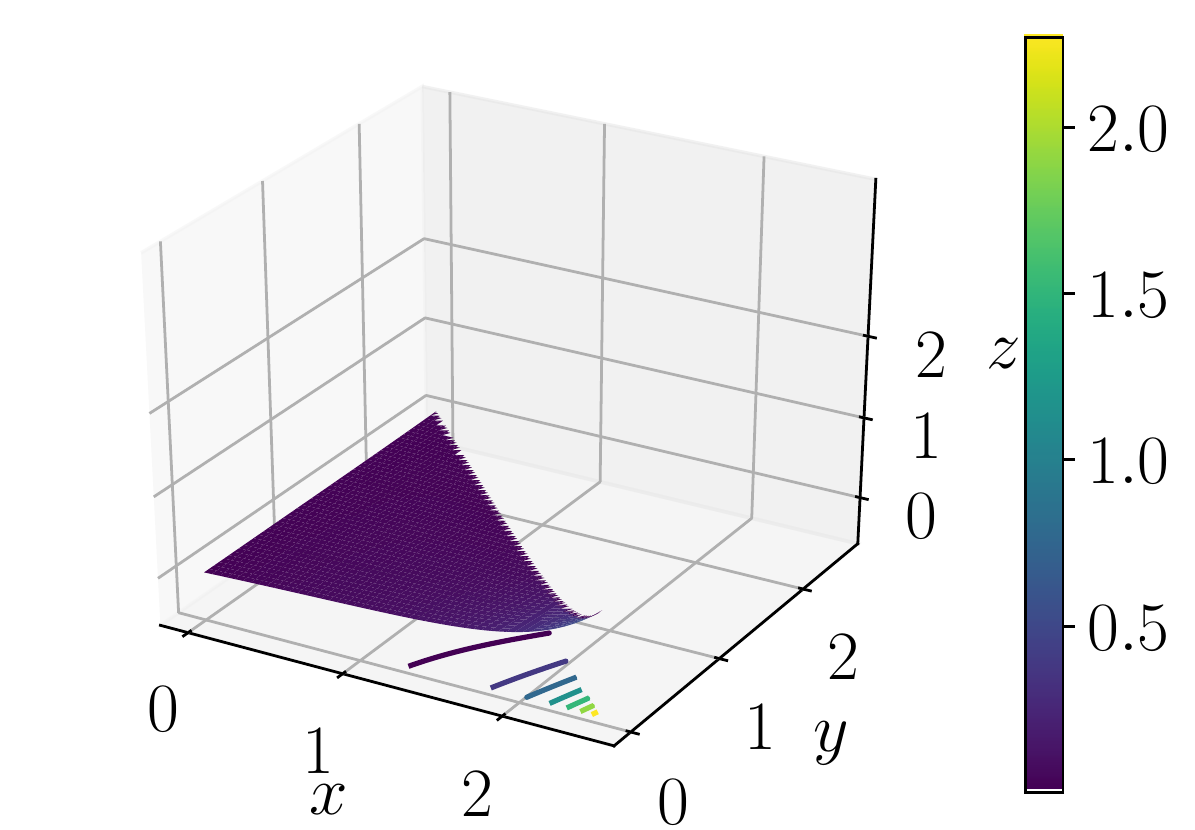} }
            \end{tabular}
        \end{tabular}
        \caption{\label{fig:test6}\textbf{Test case 6.} Bivariate sum of square approximations (of degree $n=6$) of the Motzkin polynomial. (Top left) error $\|\nabla G(\lambda_m)\|_2$ vs. iteration $m$; (Top right) time step $\tau_m$ vs. iteration $m$; (Bottom left) The Motzkin polynomial; (Bottom center) Sum of square approximation with weights $g_1 = \mu_2\,\mu_3, \ g_2 = \mu_3\,\mu_1, \ g_3 = \mu_1\,\mu_2 $ and $ g_4 = 1$, the algorithm has converged; (Bottom right) Sum of square approximation without weights ($g_1 = g_2 = g_3 = g_4 = 1$), the algorithm has not converged;}
    \end{figure} 
    
    \section{Concluding  remarks}
    
    \blue{In this paper, we reformulated the problem of computing SOS decompositions into a new nonlinear convex program. On the theoretical side we analyzed  this reformulation in detail, and particular
    the domain which guarantees convexity, 
     and showed that up to a perturbation the problem is proper strictly convex and coercive, which ensures the existence of a solution which may be explicitly approached via iterative methods. As the literature in numerical optimization is immense, we did not give an exhaustive numerical comparison of all the methods that could be used to solve our problem. Preferably, we tried to design robust methods specifically adapted to the structure of our objective function and compare it with some classical algorithms in numerical optimization. Our numerical results show that the modified Newton-Raphson algorithm is robust to compute polynomials which respect a sign condition on a given simple semi-algebraic set. However more needs to be investigated to compare with different methods and evaluate the full potential of such methods.}  Here we detail  possible domains of research  which are consequences of the multiple connections of our methods with the ones of Scientific Computing.
    \newline
    $\bullet$  It is possible to look in more details in the case $i_*\neq r_*$. It allows greater generality of the construction,
    which can be convenient for optimization purposes. In such cases, the function $G$ should be replaced by $G_\bV$.
    \newline
    $\bullet$  The technical conditions on the linear independence of the matrices $B_r$ in the multivariate case needs further examinations. In this direction there may be links with algebraic properties such as the Archimedeanity of the quadratic module associated with the weights $g_j$ (see \cite[proof of Theorem 2.14]{lasserre_2010_moments})
    or
    the condition of  Linear Independence Constraint Qualification (LICQ) \cite{nocedal_2006_numerical}.
    \newline
    $\bullet$ A C{++} implementation
    needs to be tested.
    On this basis it will be possible to couple  with codes in scientific computing 
    (such as the ones evoked in \cite{shu}
    and the references therein)
    to evaluate the gain in robustness with  the new algorithms.
    Comparisons with other established softwares like the primal-dual interior-point SDP Mosek-Yalmip package  \cite{yalmip} in Matlab will be a plus. Such benchmarks are left for future work.

    \appendix  
    \section{The asymptotic cone for univariate polynomials}\label{s:asympcone}
    
    One can obtain a much better understanding of the cone at infinity, which  exemplifies the role of the
    Lagrange interpolating polynomials.
    Given a subset $S\subset\RR^{n+1}$ one denotes by $\mathrm{coni}(S)$ the \emph{conical hull} of $S$ that is the set of linear combinations with non-negative coefficients of elements of $S$.
    The asymptotic cone  can be  constructed from the matrices
    (\ref{eq:Br1D}) or (\ref{e:hankel}) 
    in the univariate case.
    The main result  is the following, where the the Lagrange vectors are defined in 
        (\ref{e:lagrange}). 
    
    \begin{theorem}\label{t:CinfLagrange}
        The asymptotic cone of $\cD$ is generated by the Lagrange vectors 
        $L(x)$ for $0\leq x \leq 1$, that is 
        $
        C_\infty = \mathrm{coni}(\{L(x)\in\RR^{n+1}\mid x\in[0,1]\})$. 
    \end{theorem}

    We need some intermediate results in order to prove Theorem~\ref{t:CinfLagrange}. First, let us define
    $    C_\infty^1 = \{\lambda\in C_\infty \mid\quad \sum_{r=1}^{n+1}\lambda_r = 1\} \subset   C_\infty$.
    Since $\sum_{r=1}^{r_*}l_r(x_r) = 1$ for all $1\leq r\leq r_*$, one has $\sum_{r=1}^{r_*}l_r(X) = 1$. Therefore, with Lemma~\ref{l:lagrangeCinf}, we know that $\left\{L(x)\in\RR^{n+1}\mid x\in[0,1]\right\} \subset C_\infty^1$. The main point of the proof is to show that 
    $   C_\infty^1 \subset \left\{L(x)\in\RR^{n+1},\mid x\in[0,1]\right\}
    $.
    To do so  we identify $C_\infty^1$ with a subset of Borel probability measures on $[0,1]$ using the theory of  the moment problem for which an comprehensive reference is \cite{krein_1977_markov}.
    The proof of the Theorem invoked below in the proof 
    is strongly related to  the Lukacs decomposition of  Theorem~\ref{t:lukacs}.

    \begin{proposition}\label{p:moment}
        Let $\lambda\in\RR^{n+1}$. The following are equivalents: 
        a)  The vector $\lambda$ belongs to $C_\infty^1$; b) 
        There is a Borel probability measure $\sigma$ on $[0,1]$ such that 
        \begin{equation}
            \sum_{r=1}^{n+1}\lambda_r B_r = \int_{[0,1]}B(x) \mathrm{d}\sigma(x) .
            \label{e:measure}
        \end{equation}
    \end{proposition}

    \begin{proof}
        Using \eqref{e:hankel}, one can say that
        $\lambda\in C_\infty^1$ $\Longleftrightarrow$ $(s_0, \dots, s_n)$ are	 such that $H_1$ and $H_2$ are positive semidefinite matrices and $s_0=1$. By \cite[Theorem 2.3, Theorem 2.4]{krein_1977_markov}, this is equivalent to the existence of a Borel probability measure $\sigma$ such that \eqref{e:measure} holds.
    \end{proof}

    \begin{corollary}\label{cor:compact}
        The set $C_\infty^1$ is compact.
    \end{corollary}
    \begin{proof}
        Since, by Proposition~\ref{p:moment}, the $s_i$'s are moments of a Borel probability measure on $[0,1]$, one has $(s_0, \dots,s_n)\in[0,1]^{n+1}$. Therefore, since $\lambda\mapsto(s_0, \dots,s_n)$ is linear and invertible (see Lemma~\ref{lemma:ind}), $C_\infty^1$ is bounded.
    \end{proof}

    We recall that a point $\lambda$ of a convex set $C$ is said to be an extreme point (see \cite[III, Definition 2.3.1]{hiriart_1993_convex}) of $C$ if for any $\lambda_1, \lambda_2\in C$ such that $\lambda = (\lambda_1 + \lambda_2)/2$, one has $\lambda = \lambda_1 = \lambda_2$. We denote by $\mathrm{ext}(C)$  the set of extreme points of $C$.
    
    \begin{proposition}\label{p:extreme} The set of extreme points of $C_\infty^1$ is 
        $
        \mathrm{ext}(C_\infty^1) = \{L(x)\mid x\in[0,1]\} $.
    \end{proposition}
    \begin{proof}
        Let $\lambda\in\mathrm{ext}(C_\infty^1)$. Since extreme points of a convex set are located on its boundary there is a vector $\bV\neq 0$ such that $\lla\sum_{r=1}^{n+1}\lambda_r B_r \bV, \bV\rra  =0$. Let $\sigma$ be a Borel measure satisfying \eqref{e:measure} and define $q(X) = \lla B(X) \bV, \bV\rra\geq0$. One has 
        $
        \int_{[0,1]}q(x) \mathrm{d}\sigma(x)  = 0 $. 
        Since $q$ is not identically zero, the measure $\sigma$ must be supported on a subset of the finite set of roots of $q$ intersected with $[0.1]$. Since $q$ has degree $n$, $\sigma$ has the form
        $
        \sigma = \sum_{k = 1}^{n} \alpha_k \delta_{x_k} $ where $ \sum_{k = 1}^{n} \alpha_k = 1 $, $ 0 \leq \alpha_k \leq 1 $ and $ x_k\in[0,1] $, 
        for some distinct $x_1,\dots,x_n$ and where $\delta_{x_k}$ is the Dirac measure at $x_k$. Now assume that for some index $k$, $\alpha_k\in(0,1)$. Then there is $k'\neq k$ such that $\alpha_{k'}\in(0,1)$. Then let $0\leq\varepsilon <\min(\alpha_k, \alpha_{k'},1-\alpha_k, 1-\alpha_{k'})$ and define $\sigma_1 = \sigma - \varepsilon\delta_{x_k} + \varepsilon \delta_{x_{k'}}$ and $\sigma_2 = \sigma + \varepsilon\delta_{x_k} - \varepsilon \delta_{x_{k'}}$. The measures $\sigma_1$ and $\sigma_2$ are two Borel probability measures generating different sets of moments for at least some $\varepsilon$ in the range. Since $\lambda\mapsto(s_0, \dots,s_n)$ is linear and invertible there are distinct $\lambda_1, \lambda_2\in C_\infty^1$ satisfying \eqref{e:measure} for the respective measures $\sigma_1$ and $\sigma_2$ and one has $\lambda = (\lambda_1 + \lambda_2)/2$. There is a contradiction. Therefore  either $\alpha_k=0$ or $\alpha_k=1$ so  $\sigma$ must be a dirac measure at some point $x_*\in[0,1]$. Hence $\sum_{r=1}^{n+1}\lambda_r B(x_r) = B(x_*)$ so in particular
        $
        \sum_{r=1}^{n+1}\lambda_r x_r^k = x_*^k $ for any $ 0\leq k\leq n$
        which yields $\lambda = L(x_*)$.
        Conversely if $\lambda = L(x_*)$ and $\lambda = (\lambda_1 + \lambda_2)/2$, then there are probability measures $\sigma_1$ and $\sigma_2$ such that $\delta_{x_*} = (\sigma_1  + \sigma_2)/2$. Therefore $\sigma_1$ and $\sigma_2$ are supported at $x_*$ and since they have the same mass one has $\delta_{x_*} = \sigma_1 = \sigma_2$, so $\lambda\in\mathrm{ext}(C_\infty^1)$.
    \end{proof}
    
    \begin{proof}[Proof of Theorem~\ref{t:CinfLagrange}]
        Denote by $\mathrm{co}(S)$ the \emph{convex hull} of $S$, the set of linear combinations of elements of $S$ with non-negative coefficients whose sum equals $1$. By the Minkowski (or Krein-Milman) theorem \cite[III, Theorem 2.3.4]{hiriart_1993_convex},  any compact convex set is the convex hull of its extreme points, therefore $C_\infty^1 = \mathrm{co}(\mathrm{ext}(C_\infty^1))$. Remark that  $
        C_\infty = \bigcup_{t\geq0}t C_\infty^1
        $  and the result follows.
    \end{proof}

    \bibliographystyle{siamplain}
    \bibliography{bibli}
    
\end{document}